\documentclass[12pt]{amsart}
\usepackage{preamble_general}
\usepackage[
    colorlinks,
    linkcolor=blue,
    citecolor=blue,
    urlcolor=blue,
    linktocpage=true
]{hyperref}
\usepackage[nameinlink]{cleveref}
\usepackage{orcidlink}
\usepackage{bm}
\usepackage{preamble_abbreviations}
\title[small eigenvalues on K\"ahler degeneration]{Asymptotics of small eigenvalues
on degenerations of K\"ahler manifolds}
\date{\today}

\author{Junyu Cao \orcidlink{0009-0003-6875-3981}}

\begin{document}
\begin{abstract} 
    We derive the exact asymptotic rates of the small 
    eigenvalues of the Laplacian on one-parameter 
    degenerations of compact \Kahler manifolds equipped 
    with induced background metrics. 
    This generalizes a recent result of Dai and Yoshikawa 
    to higher dimensions. 
    To achieve this, we combine Li's uniform Skoda inequality 
    with the method of auxiliary Monge-Amp\`ere equations, 
    introduced by Guo--Phong--Song--Sturm--Tong and 
    adapted by Guedj--T\^o. As an application, we 
    establish estimates for degenerations 
    of compact \Kahler manifolds with reducible singular 
    fibers.
\end{abstract}
\maketitle
\setcounter{tocdepth}{2}
\tableofcontents
\section*{Introduction}
Let \(\pi\colon X \to S \simeq \D\) be a proper surjective holomorphic map from a
complex manifold \(X\) to a Riemann surface biholomorphic to the unit disk.
We assume that \(\pi\) has connected fibers of complex dimension \(n\), that
\(X_0 \coloneqq \pi^{-1}(0)\) is the unique singular fiber, and that \(X\) is
equipped with a \Kahler metric \(\omega_X\). We refer to this datum as a
one-parameter degeneration of compact \Kahler manifolds, or 
{\Kahler degeneration} for short.

In this paper, we study the asymptotics of the 
{small eigenvalues} of the Laplacian on the smooth fibers
$(\pi^{-1}(s)\eqqcolon X_s, 
\omega_{X} \mid_{X_s}\eqqcolon \omega_s)$ 
for $s \neq 0$ near the singular fiber.

In \cite{Yoshikawa1997}, Yoshikawa proved that the spectrum of the 
Laplacian on
\((X_s,\omega_s)\) varies continuously and converges, after a suitable base change and normalization, to the spectrum of the limiting central fiber as \(s \to 0\). Precisely, write
$X_0 = \sum_{\alpha} m_\alpha D_\alpha$ as the 
sum of its irreducible components $(D_\alpha)_{1
\leq \alpha \leq a}$, and let $m = \prod_{\alpha}m_\alpha$.
Consider the following commutative diagram,
\[
\begin{tikzcd}
    \widehat{F^{-1}X} \arrow[r, "\iota"] \arrow[dr, bend right=20, "\widehat{\Pi}"'] & F^{-1}X \coloneqq X \times_{\D_s} \D_t \arrow[r, "F"] \arrow[d, "\Pi"] & X \arrow[d, "\pi"] \\
    & \D_t \arrow[r, "t \longmapsto t^m"] & \D_s
\end{tikzcd}
\]
where $F^{-1}X$ is the base-change of the degeneration and
$\iota \colon \hat{F^{-1}X} \to F^{-1}X$ is the normalization
of $F^{-1}X$. We define $Z = \widehat{\Pi}^{-1}(0)$, which is a
reduced divisor. Then Yoshikawa's continuity theorem takes the following form.
\begin{theorem}[\cite{Yoshikawa1997}*{Main theorem}]
    \label[theorem]{thm:Yoshikawa_continuity}
    As $s \to 0$,
    the spectrum of the Laplacian of $(X_s, \omega_s)$
    converges to the spectrum of the Laplacian of
    $(Z_{\reg}, g_Z)$, where $g_Z$ is the restriction of the background
    metric $\iota^\ast F^\ast \omega_{X}$. We 
    remark that $F \circ \iota$ is an immersion 
    on the smooth part of $Z$, so $\iota^\ast F^\ast \omega_{X}$ is a well-defined metric
    on $Z_{\reg}$.
\end{theorem}
Although \cite{Yoshikawa1997} is stated for projective degenerations, the same
argument applies to \Kahler degenerations once one uses a uniform Sobolev
inequality for \Kahler families; see \cite{Tosatti2010}*{Lemma~3.2}.

Let $N$ be the number of irreducible components of $Z$. Because $Z_{\reg}$ has $N$ connected components of finite volume, there are exactly $N$ zero eigenvalues in the spectrum of the Laplacian of $(Z_{\reg}, g_Z)$. 
On the other hand, for $s \neq 0$, the fiber $X_s$ is connected, meaning the Laplacian of $(X_s, \omega_s)$ has exactly one zero eigenvalue, which we denote as $\lambda_0(s) = 0$. 
Let
$$0 = \lambda_0(s) < \lambda_1(s) \leq \lambda_2(s) \leq \cdots \leq \lambda_N(s)$$
be the first $N+1$ eigenvalues of the Laplacian of $(X_s, \omega_s)$. 
By Yoshikawa's continuity theorem (\Cref{thm:Yoshikawa_continuity}), exactly $N$ eigenvalues of $X_s$ must converge to $0$ as $s \to 0$.
Therefore, we have
$$\lambda_1(s), \cdots, \lambda_{N-1}(s) \to 0, \quad \lambda_N(s) \to \lambda_{N, Z} \quad \text{as } s \to 0,$$
where $\lambda_{N, Z} > 0$ is the first strictly positive eigenvalue of $(Z_{\reg}, g_Z)$.

We call $\lambda_1(s), \cdots, \lambda_{N-1}(s)$ the \textbf{small eigenvalues}
of the Laplacian on $(X_s,\omega_s)$. Thus small eigenvalues exist if and only
if $N \ge 2$. When $X_0$ is reduced, we have $Z = X_0$, so $N$ is the number
of irreducible components of $X_0$. In particular, if $\supp(X_0)$ is
reducible, then small eigenvalues occur.

Our first main result gives a logarithmic lower bound for the first positive eigenvalue.

\begin{theorem}[=\Cref{thm:lower-bound-section1}]
    \label[theorem]{thm:main-thm-lower-bound}
    Let $\lambda_1(s) > 0$ denote the first eigenvalue of the Laplacian 
    on $(X_s, \omega_s)$ for $s \neq 0$. Then, there exists 
    a uniform constant $C > 0$, independent of $s$, such that 
    \[
    C \abs{\log^{-1} \abs{s}} \leq \lambda_1(s)
    \]
    for $0 < \abs{s} \ll 1$.
\end{theorem}

When $N \ge 2$, so that $\lambda_1(s)$ is a small eigenvalue, this
\(\abs{\log^{-1}\abs{s}}\) lower bound is optimal.

\begin{theorem}[=\Cref{thm:upper-bound-general-in-section2}]
    \label[theorem]{thm:main-thm-upper-bound}
    Suppose that $N \geq 2$. Then, there exists a uniform constant
    $C' > 0$ such that 
    \[
    \lambda_1(s) \leq \cdots \leq \lambda_{N-1}(s)
    \leq C' \abs{\log^{-1} \abs{s}}
    \]
    for $0 < \abs{s} \ll 1$.
\end{theorem}

Combining \Cref{thm:main-thm-lower-bound} and 
\Cref{thm:main-thm-upper-bound} yields a complete 
characterization of the asymptotic behavior of all {small eigenvalues}.

\begin{theorem}[Main theorem]
    \label[theorem]{thm:main-thm-combined}
    Suppose a small eigenvalue exists, i.e., $N \geq 2$. Then there exist constants $C_1, C_2 > 0$ such that 
    for all $0 < \abs{s} \leq \frac{1}{2}$,
    \[
    C_1 \abs{\log^{-1} \abs{s}} \leq
    \lambda_1(s) \leq \cdots \leq \lambda_{N-1}(s)
    \leq C_2 \abs{\log^{-1} \abs{s}}.
    \]
\end{theorem}

\begin{remark}
    Due to the continuity and strict positivity of the eigenvalues 
    away from the singular fiber, the domain of validity can be naturally 
    extended from $0 < \abs{s} \ll 1$ to $0 < \abs{s} \leq \frac{1}{2}$.
\end{remark}

\begin{remark}
    The main theorem generalizes the central result of Dai and Yoshikawa
    \cite{Dai-Yoshikawa2025}*{Theorem~0.2} for degenerations of Riemann
    surfaces, where the singular fiber \(X_0\) is assumed to be reduced. In
    contrast, \Cref{thm:main-thm-combined} allows non-reduced and more general
    singular fibers, thereby resolving a question raised in
    \cite{Dai-Yoshikawa2025}*{Problem~9.1}.
\end{remark}

\begin{remark}
    In \cite{Gromov1992}, Gromov established a polynomial spectral gap,
    \(\lambda_1(s) \geq c \abs{s}^{\alpha}\), for degenerations of
    semi-algebraic submanifolds of the standard sphere \(\S^{N-1}\). Our
    logarithmic rate \(\abs{\log^{-1} \abs{s}}\) is sharper than such a
    polynomial bound. Moreover, \Cref{thm:main-thm-combined} shows that this
    logarithmic rate is optimal whenever small eigenvalues exist.
\end{remark}

We briefly sketch the proof of \Cref{thm:main-thm-combined}. For the upper
bound in \Cref{thm:main-thm-upper-bound}, we adapt the test-function
construction of Dai and Yoshikawa \cite{Dai-Yoshikawa2025}*{Section~6} and
extend it to higher dimensions in \Cref{sec:upper-bound-section}. The key
point is that we carry out the relevant estimates directly on a semistable
model of the degeneration, which allows us to remove the reducedness
assumption imposed in \cite{Dai-Yoshikawa2025}.

The lower bound in \Cref{thm:main-thm-lower-bound} is subtler. As observed in
\cite{Dai-Yoshikawa2025}*{Appendix}, the curvature of \((X_s,\omega_s)\)
blows up as \(s \to 0\), so classical Riemannian eigenvalue estimates based on
uniform lower curvature bounds are not available in our setting. We therefore
take a purely complex-analytic route. More precisely, our method combines
Li's uniform Skoda inequality for plurisubharmonic functions \cite{Li2024}
with the method of auxiliary Monge--Amp\`ere equations for estimating Green's
functions \cites{Guo-Phong-Sturm2024, Guedj-To2025}. By contrast, Dai and
Yoshikawa used analytic torsion \cite{Bismut-Bost1990} to bypass the
curvature blow-up, but that argument is specific to degenerations of Riemann
surfaces. This is why a new complex-analytic approach is needed in higher
dimensions.

This article is organized as follows. In \Cref{sec:lower-bound-section} and
\Cref{sec:upper-bound-section}, we establish the lower and upper bounds
following the strategy outlined above. In
\Cref{sec:applications-subsection}, we derive estimates for degenerating
families of plurisubharmonic functions and for solutions to Poisson equations,
valid for general singular fibers. In
\Cref{sec:KE-small-eigen-examples-subsection}, we present examples of the
{small eigenvalue} phenomenon for degenerating families equipped with
\Kahler--Einstein metrics. Finally, in \Cref{sec:NA-subsection}, we propose a
conjectural non-Archimedean interpretation of this behavior.

\subsection*{Notation and conventions}  
On a complex manifold $X$, we define 
${\di}^c = \frac{\sqrt{-1}}{2}(\bar{\partial} - \partial)$, 
so $\ddc = \sqrt{-1}\partial\bar{\partial}$.

On a \Kahler manifold $(X^n, \omega)$ with $\omega = \sqrt{-1} \sum g_{j\bar{k}} dz_j \wedge d\bar{z}_k$, we use the positive definite analyst's Laplacian $\Delta_{\omega} f = - \sum g^{j\bar{k}} \frac{\partial^2 f}{\partial z_j \partial \bar{z}_k}$. Under our normalization of $\ddc$, we have
\begin{align*}
    &-\Delta_\omega f = \operatorname{tr}_{\omega}(\ddc f) = \frac{n \ddc f \wedge \omega^{n-1}}{\omega^{n}},\\
    &\abs{\di f}^2_\omega = \operatorname{tr}_{\omega}(\di f \wedge
    \di^c f) = \frac{n \di f \wedge \di^c f\wedge \omega^{n-1}}{\omega^{n}},
\end{align*}
for every smooth function $f$.

For a closed smooth differential form $\alpha$ on a compact differential manifold,
we use $[\alpha]$ to denote its de Rham cohomology class.

We use $\D \coloneqq \{z \in \C \colon\ \abs{z} < 1\}$ to denote the unit disk in the complex plane, and $\D^\circ \coloneqq \D\setminus \{0\}$ 
the punctured unit disk. For a positive number $r > 0$, we also 
define $\D_r \coloneqq \{z \in \C \colon\ \abs{z} < r\}$ and 
$\D_r^\circ \coloneqq \D_r \setminus \{0\}$.

We use the following asymptotic analysis notation. 
Let $(X, d)$ be a metric space, fix a point $s_0 \in X$, 
and let $f$ and $g$ be functions taking values in a partially 
ordered $\R$-vector space. 
Assuming $f$ and $g$ are defined on a punctured neighborhood of 
$s_0$, we say that:
\begin{itemize}
    \item $f = O(g)$ (or $f \lesssim g$) as $s \to s_0$ if there 
    exists a constant $C > 0$, independent of $s$, such that 
    $f \leq C g$ on a sufficiently small punctured neighborhood of $s_0$.
    \item $f = \Theta(g)$ (or $f \asymp g$) as $s \to s_0$ if both 
    $f = O(g)$ and $g = O(f)$. Equivalently, this holds if there 
    exist constants $0 < C_1 \leq C_2$ such that $C_1 g \leq f \leq 
    C_2 g$ on a sufficiently small punctured neighborhood of $s_0$.
\end{itemize}

Throughout this paper, we use without quoting that every 
$1$-parameter \Kahler degeneration admits a semistable reduction 
after a proper base change; we refer to \cite{KKMSD73} for this result. 
We mention that any semistable reduction of a \Kahler degeneration 
is a \Kahler manifold.

\vspace{1em}
\noindent\textbf{Acknowledgements.}
The author thanks his advisor Valentino Tosatti for helpful discussions,
suggestions on early drafts, and his constant support. The author also
thanks Xianzhe Dai, Mikhael Gromov, and Junsheng Zhang for helpful
discussions. He is grateful to S\'ebastien Boucksom, Xianzhe Dai,
Henri Guenancia, Yang Li, Chung-Ming Pan, Freid Tong, Ken-Ichi Yoshikawa,
and Junsheng Zhang for their interest and useful feedback on an earlier draft.

\section{Lower bound}
\label[section]{sec:lower-bound-section}
In this section, we prove the lower bound of small eigenvalues
in \Cref{thm:main-thm-lower-bound}. We prove it using
rescaled uniform Skoda estimates in \Cref{sec:rescaled-skoda}
together with the method of auxiliary Monge-Amp\`ere equations
in \Cref{sec:auxiliary-MA}. Finally, the lower bound 
is proved in \Cref{sec:skoda_implies_spectral_gap}.

\subsection{A rescaled uniform Skoda estimate}
\label[subsection]{sec:rescaled-skoda}
In this section, we prove the following rescaled uniform Skoda 
inequality.
\begin{proposition}
    \label[proposition]{prop:rescaled-skoda-general-sing}
    Let $\pi \colon (X, \omega_X) \to \D_s$ be a degeneration of 
    \Kahler manifolds of complex dimension $n$. Assume $X_0$ is the 
    unique singular fiber (we allow $X_0$ to be a general singular fiber).

    Let $\beta$ be a closed semi-positive form on $X$. Assume that the restriction $\beta_s = \beta|_{X_s}$ is \Kahler for $s \neq 0$, and denote its volume by $V_{\beta_s} = \int_{X_s} \beta_s^n$. Let $\di V_{\beta_s} = \beta_s^n / V_{\beta_s}$ be the normalized volume form on $X_s$ for $s \neq 0$.

    Consider the rescaled relative \Kahler form $\tilde{\beta}_s = \frac{\beta_s}{\abs{\log \abs{s}}}$ on $X \setminus X_0$. Then the following Skoda-type estimate holds: there exist constants $\alpha, C > 0$ independent of $s$, such that for all $0 < \abs{s} \ll 1$ and for any $\vphi \in \PSH(X_s, \tilde{\beta}_s)$ with $\sup_{X_s} \vphi = 0$, we have
    \[
        \int_{X_s} \exp(-\alpha \vphi) \di V_{\beta_s} \leq C.
    \]
\end{proposition}

\begin{remark}
    In \cite{Li2024}*{Theorem~1.3}, Li proved a similar rescaled Skoda estimate for normalized Calabi-Yau measures. Although the volume forms $\di V_{\beta_s}$ differ significantly from the Calabi-Yau measures near the singular fiber, the underlying proof strategies are analogous.
\end{remark}

Before proceeding with the proof of \Cref{prop:rescaled-skoda-general-sing}, we establish the following reductions.

\begin{itemize}
    \item \textbf{Reduction to the semistable case:} We can assume that the degeneration $\pi \colon X \to \D_s$ is semistable (i.e., the central fiber $X_0$ is a reduced, simple normal crossing divisor).

    \begin{proof}[Proof of Reduction]
        By performing a semistable reduction after a suitable base change of $\pi \colon X \to \D_s$, we obtain the following commutative diagram:
        \[
        \begin{tikzcd}
            & Y  \arrow[r, "\mu"] \arrow[d, "p"] & X \arrow[d, "\pi"] \\
            & \D_t \arrow[r, "f_d"] & \D_s
        \end{tikzcd}, \quad f_d(t) = t^d.
        \]
        Here, the central fiber $Y_0$ is a reduced snc divisor in $Y$. Let $\beta' = \mu^* \beta$. The pullback $\beta'$ is a closed semi-positive form on $Y$ and restricts to a relative \Kahler form on $Y \setminus Y_0 \to \D_t^\circ$.

        Given any $\vphi \in \PSH(X_s, \tilde{\beta}_s)$ with $\sup_{X_s} \vphi = 0$, its pullback satisfies $d (\mu^* \vphi) \in \PSH(Y_t, \tilde{\beta}'_t)$, where $\tilde{\beta}'_t
        = \frac{\beta'}{\abs{\log\abs{t}}}$ is the rescaled form 
        on $Y \setminus Y_0$. Assuming the estimate holds in the semistable setting (applied to $(Y, \beta')$ and $ d (\mu^* \vphi)$), there exist uniform constants $\alpha, C > 0$ such that
        \[
        \int_{Y_t} \exp(-(\alpha d) \mu^* \vphi) \di V_{\beta'_t} \leq C.
        \]
        Since the normalized volume forms satisfy $\di V_{\beta'_t} = \mu^*\di V_{\beta_s}$, it immediately follows that
        \[
        \int_{X_s} \exp(-(\alpha d) \vphi) \di V_{\beta_s} = \int_{Y_t} \exp(-(\alpha d) \mu^* \vphi) \di V_{\beta'_t} \leq C,
        \]
        which validates the reduction.
    \end{proof}

    \item \textbf{Reduction to a globally \Kahler form:} We can assume that $\beta$ is a \Kahler form on the total space $X$, rather than merely a semi-positive form.

    \begin{proof}[Proof of Reduction]
        Suppose the result holds for any globally \Kahler form $\gamma$ on $X$. Let $\beta$ be a closed semi-positive form on $X$. Since $\pi \colon X \to \D_s$ is proper and $\gamma$ is strictly positive, there exists a constant $C_0 > 0$ such that $0 \leq \beta \leq C_0 \gamma$ restricted to a smaller disk, say $X|_{\D_{1/2}}$.

        Consequently, on each fiber $X_s$ with $\abs{s} \leq \frac{1}{2}$, the volume forms satisfy:
        \[
        \di V_{\beta_s} \leq \frac{V_{\gamma_s} C_0^n}{V_{\beta_s}} \di V_{\gamma_s} \eqqcolon C_1 \di V_{\gamma_s},
        \]
        where $C_1 > 0$ depends only on $C_0$ and the cohomology classes of $\beta$ and $\gamma$.

        Furthermore, if $\vphi \in \PSH(X_s, \frac{\beta_s}{\abs{\log \abs{s}}})$, then $\vphi \in \PSH(X_s, \frac{C_0\gamma_s}{\abs{\log \abs{s}}})$, which implies $\frac{\vphi}{C_0} \in \PSH(X_s, \frac{\gamma_s}{\abs{\log \abs{s}}})$.

        Applying our assumption for the \Kahler form $\gamma$ to the function $\frac{\vphi}{C_0}$ (which satisfies $\sup_{X_s} \frac{\vphi}{C_0} = 0$), there exist constants $\alpha, C > 0$ independent of $s$ such that for $0 < \abs{s} \ll 1$,
        \[
        \int_{X_s} \exp\left(-\alpha \frac{\vphi}{C_0}\right) \di V_{\gamma_s} \leq C.
        \]
        Using the volume bound $\di V_{\beta_s} \leq C_1 \di V_{\gamma_s}$, we deduce:
        \[
        \int_{X_s} \exp\left(- \frac{\alpha}{C_0}\vphi\right) \di V_{\beta_s} \leq C \cdot C_1.
        \]
        By setting $\alpha' = \alpha / C_0$ and $C' = C \cdot C_1$, we recover the Skoda estimate for the semi-positive form $\beta$.
    \end{proof}
\end{itemize}

Consequently, to establish \Cref{prop:rescaled-skoda-general-sing}, it suffices to prove the following simplified proposition.

\begin{proposition}
    \label[proposition]{prop:rescaled-skoda-kahler-snc}
    Let $\pi \colon \cY \to \D_s$ be a degeneration of compact \Kahler manifolds of complex dimension $n$. Assume the central fiber $\cY_0$ is the unique singular fiber and that it is a reduced simple normal crossing (snc) divisor; that is, the degeneration is semistable. 

    Let $\gamma$ be a \Kahler form on the total space $\cY$, and denote its restriction to the fiber $\cY_s$ by $\gamma_s = \gamma|_{\cY_s}$. For $s \neq 0$, let $V_{\gamma_s} = \int_{\cY_s}\gamma_s^n$ be the volume of the fiber. Since the fibers are cohomologous, $V_{\gamma_s} \equiv V$ is a uniform constant. Define the normalized volume form $\di V_{\gamma_s} = \gamma_s^n / V_{\gamma_s}$. 

    Consider the rescaled \Kahler forms $\tilde{\gamma}_s = \frac{\gamma_s}{\abs{\log \abs{s}}}$ on the fibers $\cY_s$ for $s \neq 0$. Then the following Skoda-type estimate holds: there exist constants $\alpha, C > 0$ independent of $s$, such that for all $0 < \abs{s} \ll 1$ and for any $\vphi \in \PSH(\cY_s, \tilde{\gamma}_s)$ with $\sup_{\cY_s} \vphi = 0$, we have
    \[
        \int_{\cY_s} \exp(-\alpha \vphi) \di V_{\gamma_s} \leq C.
    \]
\end{proposition}

We prove the proposition following the local setup of \cite{Li2024}*{Section~2}.
Let $E_i$ ($i\in I$) be the irreducible components of the reduced simple
normal crossing divisor $\cY_0$, and let $d_\gamma$ be the distance induced by
the ambient \Kahler metric $\gamma$.

For each subset $J\subset I$ such that
$E_J:=\bigcap_{i\in J}E_i\neq\emptyset$, Li considers the corresponding
stratum of $\cY_s$ in \cite{Li2024}*{Section~2},
\[
E_J^0
=
\{y\in \cY_s \colon d_\gamma(y,E_J)\lesssim \epsilon\}
\setminus
\bigcup_{J'\supsetneq J}
\{y\in \cY_s \colon d_\gamma(y,E_{J'})\lesssim \epsilon\}.
\]
Thus $E_{J}^0$ may be regarded as an $\epsilon$-tubular neighborhood of
$E_J$ in $\cY_s$, with the deeper strata removed. After shrinking the
base and choosing $\epsilon>0$ sufficiently small, these regions cover
$\cY_s$ for all $0<|s|\ll1$. We remark that all parameters defining
$E^0_J$ (e.g., $\epsilon$ and constants in $\lesssim$) are independent
of $s$.

Fix such a subset $J$, and write $p=|J|-1$. After reindexing the components
in $J$, we may assume $J=\{0,1,\dots,p\}$. Around any point of
\(
E_J \setminus \bigcup_{J'\supsetneq J} E_{J'}
\), semistability gives an {$E_J$-adapted coordinate chart} 
$\{z_i\}_{i=0}^n$ such that $z_j$ is a local
defining function of $E_j$ for $0\le j\le p$, and
\[
s=z_0\cdots z_p.
\]
In such a chart,
\[
\gamma \asymp \sum_{i=0}^n \sqrt{-1}\,\di z_i\wedge \di \bar z_i.
\]
Restricting to $\cY_s$, and using that $V_{\gamma_s}=V$ is independent of
$s$, we obtain locally
\[
\di V_{\gamma_s}
\asymp
\sum_{j=0}^p
\left(
\prod_{\substack{0\le i\le p\\ i\neq j}}
\sqrt{-1}\,\di z_i\wedge \di \bar z_i
\right)
\wedge
\left(
\prod_{k=p+1}^n
\sqrt{-1}\,\di z_k\wedge \di \bar z_k
\right).
\]

Following \cite{Li2024}*{Lemma~2.6}, we 
introduce {log scales} 
on $E^0_J$. Fix an {$E_J$-adapted chart} intersecting $E_J^0$
with $\C^\ast$-coordinates $z_1, \dots, z_p$ and 
$\C$-coordinates $z_{p+1}, \dots, z_n$. For 
a point $q$ on this chart, we refer to the subregion
\[
\{\frac{1}{2}\abs{z_i(q)} \lesssim \abs{z_i}
\lesssim 2 \abs{z_i(q)}, \, 1 \leq i \leq p\}
\]
as a {log scale}.

On each log scale we use the {log measure}
\[
\di \nu_{\log}
=
\left(
\prod_{i=1}^p
\sqrt{-1}\,\di \log z_i\wedge \di \log \bar z_i
\right)
\wedge
\left(
\prod_{k=p+1}^n
\sqrt{-1}\,\di z_k\wedge \di \bar z_k
\right).
\]

We now compare $\di V_{\gamma_s}$ and $\di\nu_{\log}$. Since
$s=z_0\cdots z_p$, on $\cY_s$ we have
\[
0=\di\log s=\di\log z_0+\cdots+\di\log z_p.
\]
Hence, for every $0\le j\le p$,
\[
\prod_{i=1}^p
\sqrt{-1}\,\di \log z_i\wedge \di \log \bar z_i
=
\prod_{\substack{0\le i\le p\\ i\neq j}}
\frac{\sqrt{-1}}{|z_i|^2}\,
\di z_i\wedge \di \bar z_i.
\]
Therefore,
\begin{equation}
\label{eqn:relation-volume-log-measure}
\di V_{\gamma_s}
\asymp
\sum_{j=0}^p
\left(
\prod_{\substack{0\le i\le p\\ i\neq j}}
|z_i|^2
\right)
\di\nu_{\log}
=
|s|^2
\left(
\sum_{j=0}^p |z_j|^{-2}
\right)
\di\nu_{\log}.
\end{equation}

With these preliminaries in place, the proposition follows by combining Li's
local Skoda estimate on each log scale (\cite{Li2024}*{Corollary~2.8}) with a bounded-overlap covering argument in \cite{Li2024}*{Theorem~2.9}.

\begin{proof}[Proof of \Cref{prop:rescaled-skoda-kahler-snc}]
    We choose the charts so that each point on $\cY_s$ is covered by at most $K$ log scales, where $K$ is a uniform constant independent of $s$ (see \cite{Li2024}*{Theorem~2.9}).
    
    By the local Skoda estimate in \cite{Li2024}*{Corollary~2.8}, there exist uniform constants $\alpha, A > 0$ such that on each log scale, we have
    \begin{equation}
        \label{eqn:log-skoda}
        \int_{\rm loc} \exp(-\alpha \varphi) \di \nu_{\log} \leq A \int_{\rm loc} \di \nu_{\log}.
    \end{equation}

    On each log scale, define the weight function $$W(z) \coloneqq \frac{\di V_{\gamma_s}}{\di \nu_{\log}}. $$

    From \Cref{eqn:relation-volume-log-measure},
    \[
    \di V_{\gamma_s} \asymp \underbrace{\sum_{j=0}^p \left(\prod_{\substack{i \neq j \\ 0 \leq i \leq p}} \abs{z_i}^2 \right)}_{\asymp W(z)} \di \nu_{\log},
    \]
    we obtain $W(z) \asymp \sum_{j=0}^p \left(\prod_{i \neq j, \, 0 \leq i \leq p} \abs{z_i}^2 \right) = \abs{s}^2 
    \sum_{j=0}^p\abs{z_j}^{-2} $. (Here, $\asymp$ denotes uniform equivalence independent of $s$).
    
    By definition of log scales, we have 
    \[
    1 \leq \frac{\max_{\rm loc} \abs{z_i}}{\min_{\rm loc}\abs{z_i}} \lesssim 4, \quad 1 \leq i \leq p.
    \]
    Since $\abs{z_0} = \frac{\abs{s}}{\prod_{i = 1}^p\abs{z_i}}$, it follows that 
    \[
    1 \leq \frac{\max_{\rm loc} \abs{z_0}}{\min_{\rm loc}\abs{z_0}} \lesssim 4^p.
    \]
    Thus, for the weight function on each log scale, we have 
    \begin{equation}
        \label{eqn:volume-over-log-weight-bound}
        1 \leq \frac{\max_{\rm loc} W(z)}{\min_{\rm loc} W(z)}
        \leq \frac{\sum_{j=0}^p (\min_{\rm loc} \abs{z_j})^{-2}}{\sum_{j=0}^p (\max_{\rm loc} \abs{z_j})^{-2}}
        \lesssim 4^{2p} \lesssim 4^{2n}.
    \end{equation}
    (We use $\lesssim$ to indicate a uniform upper bound independent of $s$).

    Therefore, we deduce a new local Skoda estimate with respect to the measure $\di V_{\gamma_s}$:
    \begin{align*}
        \int_{\rm loc} \exp(-\alpha \varphi) \di V_{\gamma_s}
        &\leq \max_{\rm loc} W(z) \int_{\rm loc} \exp(-\alpha \varphi) \di \nu_{\log} \\
        &\leq A \left(\max_{\rm loc} W(z)\right)\int_{\rm loc} \di \nu_{\log}
        \quad (\text{Using } 
        \cref{eqn:log-skoda})
        \\
        &= A \left(\max_{\rm loc} W(z)\right)\int_{\rm loc} \frac{\di V_{\gamma_s}}{W(z)}\\
        &\leq A \frac{\max_{\rm loc} W(z)}{\min_{\rm loc} W(z)} \int_{\rm loc} \di V_{\gamma_s}\\
        &\lesssim 4^{2n} A \int_{\rm loc} \di V_{\gamma_s} 
        \quad (\text{Using } 
        \cref{eqn:volume-over-log-weight-bound})
        \\
        &\leq A' \int_{\rm loc} \di V_{\gamma_s},
    \end{align*}
    where $A'$ is another uniform constant.

    Summing over the local Skoda estimates from all log scales, $\int_{\cY_s} \exp(-\alpha \vphi) \di V_{\gamma_s}$ is bounded by 
    \[
    A' \sum_{\text{log scale}} \int_{\rm loc} \di V_{\gamma_s} \leq A' K \int_{\cY_s} \di V_{\gamma_s} = A' K \eqqcolon C.
    \]
\end{proof}

\begin{remark}
Although \cite{Li2024}*{Section~2} is written in the projective setting, the
arguments used in \cite{Li2024}*{Lemma~2.6 and Corollary~2.8} are local on the
total space. They use only the semistable coordinate model
$s=z_0\cdots z_p$ and the uniform equivalence of the ambient \Kahler metric
with the Euclidean metric in such coordinates. Hence the same local estimates
apply to semistable degenerations of compact \Kahler manifolds.
\end{remark}

\subsection{Method of auxiliary Monge-Amp\`ere equations}
\label[subsection]{sec:auxiliary-MA}

In this subsection, we review the method of auxiliary Monge-Amp\`ere
equations introduced by Chen--Cheng \cite{Chen-Cheng2021} and Guo--Phong--Tong, and further developed by
Guo--Phong--Song--Sturm in a recent series of papers
\cites{Guo-Phong-Tong2023,Guo-Phong-Sturm2024,Guo-Phong-Song-Sturm2024,GPSS-Sobolev}.
This method is highly effective for estimating Green functions on
\Kahler manifolds. The core idea is to construct auxiliary complex
Monge-Amp\`ere equations that satisfy proper bounds, such as the uniform $L^\infty$-estimate in \Cref{thm:uniform-MA}, and whose
solutions can be used to bound solutions of the corresponding Poisson
equations; see \Cref{lem:comparing_laplace_and_MA}. For a broader
discussion of this technique and its other applications, we refer the
reader to the survey \cite{Guo-Phong-Survey}.

Our exposition closely follows the framework of Guedj--T\^o
\cite{Guedj-To2025}, which adapts naturally to our setting. Throughout
this subsection, we work under the following assumptions.

\begin{setup}
\label[setup]{setup:kahler-family-central-singular-fiber}
Let $\pi \colon X \to \D_s$ be a proper, surjective holomorphic map
with connected fibers, where $X_0 \coloneqq \pi^{-1}(0)$ is the unique
singular fiber. Assume that $X$ is equipped with a \Kahler form, and
that each regular fiber $X_s \coloneqq \pi^{-1}(s)$, for $s \neq 0$,
is an $n$-dimensional complex manifold.
\end{setup}

Compared with Setting~1.5 in \cite{Guedj-To2025}, we drop the assumption
that $X_0$ is irreducible, thus allowing an arbitrary singular central
fiber.

In this subsection, we fix a semi-positive form $\beta$ on $X$, set
$\beta_s \coloneqq \beta|_{X_s}$, and assume that
\(
V_{\beta_s} \coloneqq \int_{X_s} \beta_s^n
\)
is uniformly bounded away from $0$ and $\infty$. Let
\(
\di V_{\beta_s} \coloneqq \frac{\beta_s^n}{V_{\beta_s}}
\)
be the normalized volume form on $X_s$.

Following \cite{Guedj-To2025}*{Definition~1.6}, we introduce the
following class. We fix a positive $\delta \in (0, 1]$. Here a relative \Kahler form on
$\pi^{-1}(\D_\delta)\setminus X_0$ means a smooth real $(1,1)$-form
whose restriction to each fiber $X_s$, $s \in \D_\delta^\circ$, is a
\Kahler form.

\begin{definition}
\label[definition]{dfn:skoda_class}
Fix $p > 1$ and $B,C,\alpha > 0$, $\delta \in (0,1)$. We let
$\cK_{\rm Skoda}((X, \beta),p,B,C,\alpha,\delta)$ denote the set of all
relative \Kahler forms $\theta$ on $\pi^{-1}(\D_\delta)\setminus X_0$
such that:
\begin{enumerate}
\item For every $s \in \D_\delta^\circ$, the restriction
$\theta_s \coloneqq \theta|_{X_s}$ satisfies the $L^p$-condition on
its volume density, namely
\[
\int_{X_s} f_s^p \,\di V_{\beta_s} \le B,
\qquad
\frac{\theta_s^n}{V_{\theta_s}} = f_s \,\di V_{\beta_s},
\]
where $V_{\theta_s} \coloneqq \int_{X_s} \theta_s^n$.

\item For every $s \in \D_\delta^\circ$ and every
$\vphi_s \in \PSH(X_s,\theta_s)$ with $\sup_{X_s}\vphi_s = 0$, we
have the uniform Skoda estimate
\[
\int_{X_s} \exp(-\alpha \vphi_s)\,\di V_{\beta_s} \le C.
\]
\end{enumerate}
\end{definition}

As a consequence of the rescaled Skoda estimate in
\Cref{prop:rescaled-skoda-general-sing}, we obtain the following
proposition.

\begin{proposition}
\label[proposition]{prop:rescaled-in-the-class}
Let $\beta$ be the background \Kahler metric on $X$. By
\Cref{prop:rescaled-skoda-general-sing}, the rescaled relative form
$\tilde{\beta}$ on $\pi^{-1}(\D_\delta)\setminus X_0$, defined by
\[
\tilde{\beta}|_{X_s} \coloneqq \tilde{\beta}_s
\coloneqq \frac{\beta_s}{\abs{\log \abs{s}}},
\]
belongs to $\cK_{\rm Skoda}((X, \beta),p,1,C,\alpha,\delta)$ for any $p>1$ and
for some $C,\alpha>0$, $\delta \in (0,1)$.
\end{proposition}

\begin{proof}
We shrink the disk to $\D_\delta^\circ$ so \Cref{prop:rescaled-skoda-general-sing} applies. For each $s \in \D_\delta^\circ$, the volume density $f_s$ of
$\tilde{\beta}_s$ with respect to $\di V_{\beta_s}$ is equal to $1$,
since
\[
\frac{\tilde{\beta}_s^n}{V_{\tilde{\beta}_s}}
=
\frac{\beta_s^n}{V_{\beta_s}}
=
\di V_{\beta_s}.
\]
Hence
\( \int_{X_s} f_s^p \,\di V_{\beta_s} = 1 \) 
and the condition in \Cref{dfn:skoda_class}(1)
holds for $B = 1$ and any $p > 1$.
The Skoda inequality in \Cref{dfn:skoda_class}(2) follows directly from
\Cref{prop:rescaled-skoda-general-sing}.
\end{proof}

\begin{remark}
Under \Cref{setup:kahler-family-central-singular-fiber}, assume in
addition that $X_0$ is reduced and irreducible. Fix $\delta \in (0,1)$.
Let $\theta$ be a relative \Kahler form on
$\pi^{-1}(\D_\delta)\setminus X_0$. If the following two conditions
hold:
\begin{itemize}
\item there exist $p>1$ and $B>0$ such that
\[
\int_{X_s}
\left(\frac{\theta_s^n}{V_{\theta_s}\,\di V_{\beta_s}}\right)^p
\di V_{\beta_s}
\le B
\]
for all $s \in \D_\delta^\circ$;

\item there exists $A>0$ such that $[\theta_s] \le A[\beta_s]$ for all
$s \in \D_\delta^\circ$;
\end{itemize}
then \cite{Guedj-To2025}*{Theorem~1.8} yields
\[
\theta \in \cK_{\rm Skoda}((X, \beta),p,B,C,\alpha,\delta),
\]
where $\alpha = \alpha(n,p,A,B)$ and $C = C(\alpha,n,p,A,B)$.
\end{remark}

By \cite{DNGG2022}*{Theorem~A}, we have the uniform estimate for 
Monge-Amp\`ere equations.
\begin{theorem}
    \label[theorem]{thm:uniform-MA}
    Fix $p > 1$ and $B, C, \alpha > 0, \delta\in (0, 1)$. Let $\theta \in \cK_{\rm Skoda}((X, \beta), p, B, C, \alpha, \delta)$. Assume that there exists $\vphi_s \in \PSH(X_s, \theta_s)
    \cap L^\infty(X_s)$,  $p' > 1$ and $B' > 0$ independent of $s \in \D^\circ_{\delta}$
    such that
    \[
    \frac{1}{V_{\theta_s}}(\theta_s + \ddc \vphi_s)^n = g_s \di V_{\beta_s},
    \]
    with $\int_{X_s} g_s^{p'} \di V_{\beta_{s}} \leq B'$. Then 
    $\mathrm{Osc}_{X_s}(\vphi_s) \leq L  = L(p', B', C, \alpha, n)$.
\end{theorem}

\begin{proof}
    Since $\theta \in \cK_{\rm Skoda}((X, \beta), p, B, C, \alpha, \delta)$, 
    it follows directly from \cite{DNGG2022}*{Theorem~A}. Using the notations
    therein, we take 
    \[
    X \coloneqq X_s, \quad \omega \coloneqq \theta_s, \quad 
    \nu \coloneqq \di V_{\beta_s}, \quad \mu \coloneqq g_s \nu.
    \]
\end{proof}

The following comparison lemma is the key ingredient in the method of
auxiliary Monge-Amp\`ere equations.

\begin{lemma}[\cite{Guedj-To2025}*{Proposition~1.4}]
\label[lemma]{lem:comparing_laplace_and_MA}
Suppose that $(X^n,\omega)$ is a compact \Kahler manifold of complex
dimension $n$. Fix $t > 0$, $p > 1$, and
$0 \le f \in L^{np}(\omega^n)$. Let $v$ be the unique bounded
$\omega$-sh function, and let $\vphi$ be the unique bounded
$\omega$-psh function, satisfying
\[
(\omega+\ddc v)\wedge \omega^{n-1}
=
e^{tv}f\,\omega^n
\quad \text{and}
\quad
(\omega+\ddc \vphi)^n
=
e^{nt\vphi}f^n\,\omega^n.
\]
Then $\vphi \le v$.
\end{lemma}

For $\theta \in \cK_{\rm Skoda}((X, \beta),p,B,C,\alpha,\delta)$, we obtain the
following Laplacian estimates on $(X_s,\theta_s)$ for
$s \in \D_\delta^\circ$.

\begin{lemma}[\cite{Guedj-To2025}*{Lemma~2.1}]
Fix $a>0$, and let $v$ be a quasi-subharmonic function on $X_s$ such
that
\(
\Delta_{\theta_s} v \ge -a.
\)
Then
\[
\sup_{X_s} v
\le
L\left(
a+\frac{1}{V_{\theta_s}}\int_{X_s}\abs{v}\,\theta_s^n
\right),
\]
where $L=L(n,p,B,C,\alpha)>0$ depends only on $n,p,B,C,\alpha$.
\end{lemma}

\begin{proof}
The proof of \cite{Guedj-To2025}*{Lemma~2.1} uses only the Skoda estimate in \Cref{dfn:skoda_class}(2), the uniform
$L^\infty$-estimate in \Cref{thm:uniform-MA}, and the comparison lemma
\Cref{lem:comparing_laplace_and_MA}. Since these ingredients are
available for $\theta \in \cK_{\rm Skoda}(X,p,B,C,\alpha,\delta)$, the
same argument applies.
\end{proof}

\begin{proposition}[\cite{Guedj-To2025}*{Proposition~2.2}]
\label[proposition]{prop:Laplacian-estimate-by-skoda}
Let $u$ be a continuous function on $X_s$ such that
\(
\int_{X_s} u\,\theta_s^n = 0\),
\(
\abs{\Delta_{\theta_s} u} \le 1.
\)
Then
\[
\norm{u}_{L^\infty(X_s)} \le L,
\]
where $L=L(n,p,B,C,\alpha)>0$ depends only on $n,p,B,C,\alpha$.
\end{proposition}

\begin{proof}
The proof of \cite{Guedj-To2025}*{Proposition~2.2} relies on the
previous lemma together with
\Cref{thm:uniform-MA} and \Cref{lem:comparing_laplace_and_MA}. Hence the same
argument applies in the present setting.
\end{proof}

\begin{remark}
\Cref{prop:Laplacian-estimate-by-skoda} is also proved by
Guo--Phong--Song--Sturm in \cite{Guo-Phong-Song-Sturm2024B}.
\end{remark}

\subsection{Proof of the lower bound of {small eigenvalues}}
\label[subsection]{sec:skoda_implies_spectral_gap}

We prove the lower bound of {small eigenvalues} in the following, which is part of our main theorem  
\Cref{thm:main-thm-lower-bound}.
\begin{theorem}
    \label[theorem]{thm:lower-bound-section1}
    Let \(\pi\colon (X, \omega_X) \to \D_s\) be a degeneration of \Kahler manifolds of complex dimension \(n\). Suppose
    $X_0$ is the unique singular fiber which may be 
    reducible and non-reduced. Let $\omega_s \coloneqq \omega_X\mid_{X_s}$ be the restriction of the background \Kahler form 
    $\omega_X$ on each regular fiber $X_s$ ($s \neq 0$), then
    there is a positive constant $C > 0$ such that
    \[
    C \abs{\log^{-1} \abs{s}} \leq \lambda_1(X_s, \Delta_{\omega_s}),
    \]
    for $0 < \abs{s} \ll 1$.
\end{theorem}
\begin{proof}
    Let $\Phi_s$ be the normalized eigenfunction with the first non-zero 
    eigenvalue of the Laplacian on $(X_s, \omega_s)$, i.e., 
    \[
    \Delta_{\omega_s} \Phi_s = \lambda_1(X_s, \Delta_{\omega_s}) \Phi_s,
    \quad \norm{\Phi_s}_{L^2(X_s, \omega_s^n)} = 1.
    \]

    Since $\lambda_1(X_s, \Delta_{\omega_s})$ is continuous
    in $s \in \D$, we have $0 \leq \lambda_1(X_s, \Delta_{\omega_s})
    \leq \lambda$ for some $0 < \lambda < \infty$, $0 < \abs{s}
    \leq \frac{1}{2}$.

    Recall that $(X_s, \omega_s)$ satisfies a uniform Sobolev 
    inequality as it constitutes a family of minimal submanifolds (see \cite{Tosatti2010}*{Lemma~3.2}, \cite{DNGG2022}*{Proposition~3.8}). For all $s \neq 0$ and $u \in C^\infty(X_s)$, there 
    is a constant $C_{\rm Sob} > 0$ independent of $s$ and $u$ such that
    \begin{align}
    \label{eqn:uniform_sobolev}
    \begin{split}
    \left( \int_{X_s} |u|^{2\nu} \omega_s^n \right)^{\frac{1}{\nu}} &\leq C_{\rm Sob}^2 
    \left(\int_{X_s} |\nabla u|^2_{\omega_s} \omega_s^n +\int_{X_s} |u|^2 \omega_s^n\right),
    \\ 
    &
    \left(\nu = \frac{n}{n - 1} \quad \text{when }n\geq 2\right);\\
    \left( \int_{X_s} |u|^{4} \omega_s^n \right)^{\frac{1}{2}} &\leq C_{\rm Sob}^2 
    \left(\int_{X_s} |\nabla u|^2_{\omega_s} \omega_s^n + \int_{X_s} |u|^2 \omega_s^n\right), 
    \quad \text{when } n = 1.
    \end{split}
    \end{align}

    From a standard Moser iteration (see 
    \cite{Petersen2016}*{Theorem 9.2.7}), we have 
    \[
    V_{\omega_s}^{-1/2} = 
    V_{\omega_s}^{-1/2}\norm{\Phi_s}_{L^2(X_s, \omega_s^n)} \leq  \norm{\Phi_s}_{L^\infty(X_s)} \leq C'
    \]
    for a uniform constant $C' = C'(C_{\rm Sob}, n, \lambda) > 0$. Indeed, we have 
    \begin{equation}
        \label{eqn:eigen-infinity-norm}
        \norm{\Phi_s}_{L^\infty(X_s)}
        \leq 
        \exp\left(C_{\rm Sob}
        \frac{\sqrt{\nu \lambda}}{\sqrt{\nu} - 1}\right)
        \norm{\Phi_s}_{L^2(X_s, \omega_s^n)} .
    \end{equation}

    Set $\tilde{\beta}_s = \frac{\omega_s}{\abs{\log \abs{s}}}$. Then $\tilde{\beta} \in \cK_{\rm Skoda}((X, \omega_X), p, 1, 
    C_{\rm Skoda}, \alpha_{\rm Skoda}, \delta)$
    for any $p > 1$, for some $0 < \delta < 1$ and for positive 
    constants $C_{\rm Skoda}, \alpha_{\rm Skoda}$ in the uniform Skoda inequality by \Cref{prop:rescaled-in-the-class}.
    
    Note that 
    \[
    \Delta_{\tilde{\beta}_s} \Phi_s = \abs{\log \abs{s}} \Delta_{\omega_s} \Phi_s
    = \left(\abs{\log \abs{s}} \cdot \lambda_1(X_s, \Delta_{\omega_s})\right) \Phi_s,
    \]
    and
    \[
    \abs{\Delta_{\tilde{\beta}_s} \frac{\Phi_s}{C'\abs{\log \abs{s}}
    \cdot \lambda_1(X_s, \Delta_{\omega_s}) }}
    = \abs{\frac{\Phi_s}{C'}} \leq 1.
    \]

    Applying \Cref{prop:Laplacian-estimate-by-skoda}, note that 
    $\int_{X_s} \Phi_s \tilde{\beta}_s^n = \abs{\log^{-n}\abs{s}} \int_{X_s} \Phi_s \omega_s^n = 0$,
    we have 
    \[
    \norm{\frac{\Phi_s}{C'\abs{\log \abs{s}}
    \cdot \lambda_1(X_s, \Delta_{\omega_s}) }}_{L^\infty(X_s)}
    \leq C''
    \]
    for $0 < \abs{s} < \delta$ and a uniform $C'' > 0$. 
    
    Therefore, we obtain 
    \[
    \lambda_1(X_s, \Delta_{\omega_s})
    \geq \norm{\frac{\Phi_s}{C' C'' \abs{\log \abs{s}}}}_{L^\infty(X_s)}
    \geq C \abs{\log^{-1} \abs{s}}
    \]
    for a uniform $C = (C' C'' V_{\omega_s}^{1/2})^{-1} > 0$.
\end{proof}

\section{Upper bound}
\label[section]{sec:upper-bound-section}

The main theorem of this section is the following.
\begin{theorem}
    \label{thm:stable-degeneration-upper-bound}
    Let $\pi \colon \mathcal{Y} \to \D_s$ be a degeneration 
    of compact \Kahler manifolds of complex dimension $n$. Suppose $\cY_0$ is the unique
    singular fiber and $\cY_0$ is a reduced divisor with
    simple normal crossings, i.e., the degeneration is 
    semistable. 

    Let $\beta$ be a closed smooth semi-positive
    form on $\cY$ such that its restriction on each regular 
    fiber $\cY_s\eqqcolon \pi^{-1}(s)$ $(s \neq 0)$ is \Kahler. Let 
    \[
    N(\beta, \cY) \coloneqq \# \left\{ D \text{ an irreducible component of } \cY_0 \;\middle|\; \int_D \beta^n > 0 \right\}.
    \]
    We write $N \coloneqq N(\beta, \cY)$ and 
    assume $N \geq 2$.

    Then for $0 < \lambda_1(s) \leq \cdots \leq \lambda_{N-1}(s)$
    the first $N - 1$ non-zero eigenvalues of the Laplacian of $(\cY_s, \beta_s\eqqcolon \beta\mid_{\cY_s})$, there
    exists a constant $C > 0$ such that 
    \[
    \lambda_1(s) \leq \cdots \leq \lambda_{N-1}(s) 
    \leq \frac{C}{\log\abs{s}^{-1}}
    \]
    for $0 < \abs{s} \leq \frac{1}{2}$.
\end{theorem}

In the first subsection, we show that \Cref{thm:stable-degeneration-upper-bound} implies the upper bound 
of {small eigenvalues} of general degenerations in \Cref{thm:main-thm-upper-bound}. 

In the remaining part of this section,
we prove \Cref{thm:stable-degeneration-upper-bound} following
the ideas in \cite{Yoshikawa1997} and 
\cite{Dai-Yoshikawa2025}*{Section~6}. 
The method is to construct suitable test functions on $\cY_s$ and apply the min-max 
principle to get the upper bound.  

\subsection{Proof of the upper bound of {small eigenvalues}}
Assuming \Cref{thm:stable-degeneration-upper-bound}, we prove the upper 
bound of {small eigenvalues} in \Cref{thm:main-thm-upper-bound}.

We recall the construction of $Z$ in \Cref{thm:Yoshikawa_continuity}. Note that the number of irreducible
components of $Z$ minus $1$ is equal to the number of small eigenvalues.

Let \(\pi\colon (X, \omega_X) \to \D_s\) be a degeneration of \Kahler
manifolds of complex dimension \(n\). Write the central fiber as
\(
X_0 = \sum_{\alpha=1}^a m_\alpha D_\alpha,
\)
where the \(D_\alpha\) are its irreducible components, and let
\(
m = \prod_{\alpha=1}^a m_\alpha.
\)

Consider the base change and normalization given by the following
commutative diagram:
\begin{equation}
    \label{diagram:diagram-Z}
    \begin{tikzcd}
    \widehat{F^{-1}X} \arrow[r, "\iota"] \arrow[dr, bend right=20, "\widehat{\Pi}"']
    & F^{-1}X \coloneqq X \times_{\D_s} \D_t \arrow[r, "F"] \arrow[d, "\Pi"]
    & X \arrow[d, "\pi"] \\
    & \D_t \arrow[r, "t \mapsto t^m"']
    & \D_s
\end{tikzcd}
\end{equation}
We define the reduced divisor
\(
Z \coloneqq \widehat{\Pi}^{-1}(0)
\)
as the central fiber of the normalized space.

We prove a lemma which calculates the number of irreducible
components of \(Z\) by passing to a semistable reduction of
\(\pi\colon (X,X_0)\to (\D,0)\).

\begin{lemma}
    \label[lemma]{lem:counting-irrd-cpnt-of-Z}
    For any semistable reduction \(p \colon \cY \to \D_t\) of \(\pi\)
    over a finite base change, fitting into the commutative diagram
    \[
    \begin{tikzcd}
        & \cY  \arrow[r, "\mu"] \arrow[d, "p"] & X \arrow[d, "\pi"] \\
        & \D_t \arrow[r, "f_d"] & \D_s
    \end{tikzcd}
    \]
    where \(f_d(t)=t^d=s\), and the induced map
    \((\mu,p)\colon \cY\to X\times_{\D_s}\D_t\) is an isomorphism over
    \(\D_t^\circ\), we define the following invariant of the
    degeneration:
    \begin{align*}
        N(X)
        &\coloneqq N(\mu^*\omega_X \eqqcolon \beta_\cY,\cY) \\
        &= \#\left\{
        D \subseteq \cY_0
        \;\middle|\;
        D \text{ is an irreducible component and }
        \int_D \beta_\cY^n > 0
        \right\}.
    \end{align*}

    Then
    \begin{itemize}
        \item \(N(X)\) is independent of the choice of semistable
        reduction,
        \item \(N(X)\) is equal to the number of irreducible components
        of \(Z\), which is denoted by \(N\).
    \end{itemize}
\end{lemma}

\begin{proof}
We first show the birational invariance. Let
\(\mu_i\colon \cY_i\to X\), \(i=1,2\), be two
semistable reductions over the same base change \(s=t^d\), and set
\(\beta_i=\mu_i^*\omega_X\). Let \(p_i\colon \cY_i\to\D_t\) be the structure maps.

Since the induced maps \((\mu_i,p_i)\colon \cY_i\to X\times_{\D_s}\D_t\)
are isomorphisms over \(\D_t^\circ\), the varieties \(\cY_1\) and
\(\cY_2\) are birational over \(X\).
We can
choose a common log resolution of pairs $(\cY_1, (\cY_1)_0)$ and $(\cY_2,(\cY_2)_0)$ over \(X\), namely, the following diagram:
\[
\begin{tikzcd}
& \cW \arrow[dl, "\phi_1"'] \arrow[dr, "\phi_2"] & \\
\cY_1 && \cY_2 .
\end{tikzcd}
\]
Then \(\phi_1^*\beta_1=\phi_2^*\beta_2\). For an irreducible component
\(G\subset \cW_0\), either \(G\) is exceptional over \(\cY_i\), in which
case \(\dim \phi_i(G)<n\) and the projection formula gives
\[
\int_G \phi_i^*\beta_i^n=0,
\]
or \(G\) is the strict transform of a unique irreducible component
\(D\subset (\cY_i)_0\), in which case
\[
\int_G \phi_i^*\beta_i^n=\int_D \beta_i^n .
\]
Thus the components with positive \(\beta\)-volume are preserved under
passing to a common log resolution. Hence
\[
N(\beta_1,\cY_1)=N(\beta_2,\cY_2)
\]
for semistable reductions over the same base change.

The same argument shows invariance under further ramified base change.
Indeed, let \(p\colon\cY\to\D_t\) be semistable and pull it back by
\(u\mapsto u^q\). If \(p'\colon\cY'\to\D_u\) is a semistable reduction
of \(\cY\times_{\D_t}\D_u\), and \(\psi\colon\cY'\to\cY\) is the induced map,
then \(\beta_{\cY'}=\psi^*\beta_\cY\). Every component of \((\cY')_0\) is
either the strict transform of a component of \(\cY_0\), or is exceptional
over an intersection stratum of \(\cY_0\). The exceptional components have
zero \(\beta\)-volume by the projection formula, while strict transforms
have the same volume as the original components. Therefore
\[
N(\beta_{\cY'},\cY')=N(\beta_\cY,\cY).
\]
Given two arbitrary semistable reductions of degrees \(d_1\) and \(d_2\),
we pass both to the common base change of degree
\(\ell=\operatorname{lcm}(d_1,d_2)\). The preceding paragraph and the
same-base-change invariance then imply that \(N(X)\) is independent of
the semistable reduction.

It remains to identify this number with the number of irreducible
components of \(Z\). Write
\[
Z=\widehat{\Pi}^{-1}(0)=\bigcup_{j=1}^N Z_j .
\]
After a further base change \(u\mapsto t=u^q\), take a semistable
reduction
\[
\nu\colon \cY\to \widehat{F^{-1}X}\times_{\D_t}\D_u .
\]
Let \(r\colon\cY\to\widehat{F^{-1}X}\) be the induced map and
\(h \coloneqq F \circ \iota \colon\widehat{F^{-1}X}\to X\) be the composition of maps appearing in \cref{diagram:diagram-Z}. Taking \(\mu=h\circ r\), we have
\[
\beta_\cY=\mu^*\omega_X=r^*h^*\omega_X,
\]
and we compute $N(X)$ using this semistable model,
\[
N(X)=N(\beta_\cY,\cY).
\]

Since \(\widehat{F^{-1}X}\) is normal and \(Z\) is reduced, the pullback
family is smooth at the generic point of each \(Z_j\). Hence the
irreducible components of \(\cY_0\) are precisely the strict transforms
\(\widetilde Z_j\), together with \(\nu\)-exceptional divisors. The latter
map to subsets of dimension \(<n\), so they have zero \(\beta_\cY\)-volume.
Since \(h\) is finite over the central fiber, for each \(j\), the map
\[
h_j:=h|_{Z_j}\colon Z_j\to X_0
\]
has image an irreducible component \(D_{\alpha(j)}\subset X_0\) and is
generically finite of degree \(e_j\ge 1\). Therefore
\[
\int_{\widetilde Z_j}\beta_\cY^n
=
e_j\int_{D_{\alpha(j)}}\omega_X^n
>0.
\]
Thus the components of \(\cY_0\) with positive \(\beta_\cY\)-volume are
exactly \(\widetilde Z_1,\dots,\widetilde Z_N\). Consequently
\[
N(X)=N(\beta_\cY,\cY)=N,
\]
as claimed.
\end{proof}

Then we can prove the upper bound of the main theorem.
\begin{theorem}
    \label[theorem]{thm:upper-bound-general-in-section2}
    Let \(\pi\colon (X, \omega_X) \to \D_s\) be a degeneration of \Kahler manifolds of complex dimension \(n\). Suppose
    $X_0$ is the unique singular fiber which may be 
    reducible and non-reduced. Let $N$ be the number 
    of irreducible components of $Z$, which is also 
    $N(X)$ by \Cref{lem:counting-irrd-cpnt-of-Z}. Then we have 
    $N - 1$ small eigenvalues $0 < \lambda_1(s) \leq \cdots \leq \lambda_{N-1}(s)$ of the Laplacian on $(X_s, \omega_X\mid_{X_s})$.

    If $N \geq 2$, then there is a uniform constant
    $C > 0$ such that 
    \[
    \lambda_1(s) \leq \cdots \leq \lambda_{N-1}(s)
    \leq C \abs{\log^{-1} \abs{s}}
    \]
    for $0 < \abs{s} \ll 1$.
\end{theorem}
\begin{proof}
    We use the semistable diagram in \Cref{lem:counting-irrd-cpnt-of-Z},
    \[
    \begin{tikzcd}
        & \cY  \arrow[r, "\mu"] \arrow[d, "p"] & X \arrow[d, "\pi"] \\
        & \D_t \arrow[r, "f_d"] & \D_s
    \end{tikzcd}
    \]
    Set $\beta \coloneqq \mu^* \omega_X$. Then $(\cY_t, \beta_t)$ is isometric to $(X_{f_d(t)}, \omega_{f_d(t)})$
    for $t \neq 0$.
    
    By \Cref{thm:stable-degeneration-upper-bound}, we have 
    \[
    0 < \lambda_1(\cY_t, \Delta_{\beta_t}) \leq \cdots 
    \leq \lambda_{N(\beta,\cY)-1}(\cY_t, \Delta_{\beta_t})
    \leq \frac{C_1}{\log\abs{t}^{-1}}
    \]
    for some $C_1 > 0$ and $\abs{t} \leq \frac{1}{2}$. 
    
    By \Cref{lem:counting-irrd-cpnt-of-Z}, we have 
    $N(\beta,\cY) = N(X) = N$. So 
    \[
    0 < \lambda_1(X_{f_d(t)}, \Delta_{\omega_{f_d(t)}}) \leq \cdots 
    \leq \lambda_{N-1}(X_{f_d(t)}, \Delta_{\omega_{f_d(t)}})
    \leq \frac{C_1}{\log\abs{t}^{-1}}.
    \]

    Since \(s= f_d(t) = t^d\), we have
    \[
    \log\abs{t}^{-1}=\frac{1}{d}\log\abs{s}^{-1}.
    \]
    Hence, after replacing \(C_1\) by \(C=dC_1\),
    \[
    \lambda_1(s) \leq \cdots \leq \lambda_{N-1}(s)
    \leq \frac{C}{\log\abs{s}^{-1}}
    \]
    for \(0<\abs{s} \leq \frac{1}{2^d}\).
\end{proof}

\subsection{A quantitative retraction à la Dai--Yoshikawa}

We work under the geometric assumptions of \Cref{thm:stable-degeneration-upper-bound}. Thus
\(\pi \colon \mathcal{Y} \to \mathbb{D}\) is a semistable degeneration of compact K\"ahler manifolds of
complex dimension \(n\), the central fiber \(\mathcal{Y}_{0}\) is a reduced simple normal crossings divisor,
and \(\gamma\) is a fixed K\"ahler metric on \(\mathcal{Y}\). We write
\[
\mathcal{Y}_{0} = \sum_{i=1}^{a} D_{i},
\qquad
\Sing(\mathcal{Y}_{0}) = \bigcup_{1 \le i < j \le a} (D_{i} \cap D_{j}),
\]
and
\[
\mathcal{Y}_{0}^{\mathrm{reg}} := \mathcal{Y}_{0} \setminus \Sing(\mathcal{Y}_{0}),
\]
and we denote by \(d_{\gamma}\) the distance induced by \(\gamma\). For a subset \(A \subset \mathcal{Y}\) and
\(r > 0\), we set
\[
B_{\gamma}(A,r) := \{ y \in \mathcal{Y} \mid d_{\gamma}(y,A) < r \},
\quad
\overline{B}_{\gamma}(A,r) := \{ y \in \mathcal{Y} \mid d_{\gamma}(y,A) \leq r \}
.
\]

\begin{proposition}[Quantitative retraction]
\label[proposition]{prop:quantitative-retraction}
Fix an integer \(\nu \ge n\), and set
$\epsilon(s) := 2 |s|^{\frac{1}{4\nu}}$.
Then, after shrinking \(\mathbb{D}\) if necessary, there exists a family of diffeomorphisms
\[
F_{s} \colon \mathcal{Y}_{0}^{\mathrm{reg}} \setminus \overline{B}_{\gamma}(\Sing(\mathcal{Y}_{0}), \epsilon(s))
\longrightarrow
F_{s}\bigl(\mathcal{Y}_{0}^{\mathrm{reg}} \setminus \overline{B}_{\gamma}(\Sing(\mathcal{Y}_{0}), \epsilon(s))\bigr)
\subset \mathcal{Y}_{s}
\]
with the following properties.
\begin{enumerate}
\item \(F_{0} = \mathrm{id}_{\mathcal{Y}_{0}^{\mathrm{reg}}}\).
\item For every \(z\) in the domain of \(F_{s}\),
\[
d_{\gamma}(F_{s}(z), z) \le K_{1} |s|^{\frac{3}{4}}.
\]
\item If \(\beta\) is a smooth differential form on \(\mathcal{Y}\), then
\[
\norm{F_{s}^{*}\beta_{s} - \beta_{0}}_{L^{\infty}(\mathcal{Y}_{0}^{\mathrm{reg}} \setminus
\overline{B}_{\gamma}(\Sing(\mathcal{Y}_{0}), \epsilon(s)))} \le K_{2,\beta} |s|^{\frac{1}{2}}.
\]
\item Let \(\gamma_{s} := \gamma|_{\mathcal{Y}_{s}}\), and let \((F_{s})_{*}\chi\) denote the push-forward of a
test function \(\chi\) by \(F_{s}\), extended by \(0\) outside the image of \(F_{s}\). Then for all
\(\chi,\chi' \in C_{0}^{\infty}(\mathcal{Y}_{0}^{\mathrm{reg}} \setminus \overline{B}_{\gamma}(\Sing(\mathcal{Y}_{0}), \epsilon(s)))\),
\[
\left| \bigl( (F_{s})_{*}\chi, (F_{s})_{*}\chi' \bigr)_{L^{2}(\mathcal{Y}_{s},\gamma_{s})}
- (\chi,\chi')_{L^{2}(\mathcal{Y}_{0}^{\mathrm{reg}},\gamma_{0})} \right|
\le
K_{3} |s|^{\frac{1}{2}} \norm{\chi}_{L^{2}(\mathcal{Y}_{0}^{\mathrm{reg}},\gamma_{0})}
\norm{\chi'}_{L^{2}(\mathcal{Y}_{0}^{\mathrm{reg}},\gamma_{0})},
\]
and
\[
\left|
\norm{ d\bigl((F_{s})_{*}\chi\bigr) }_{L^{2}(\mathcal{Y}_{s},\gamma_{s})}^{2}
-
\norm{ d\chi }_{L^{2}(\mathcal{Y}_{0}^{\mathrm{reg}},\gamma_{0})}^{2}
\right|
\le
K_{3} |s|^{\frac{1}{2}} \norm{ d\chi }_{L^{2}(\mathcal{Y}_{0}^{\mathrm{reg}},\gamma_{0})}^{2}.
\]
In particular,
\[
\left|
\norm{ d\bigl((F_{s})_{*}\chi\bigr) }_{L^{2}(\mathcal{Y}_{s},\gamma_{s})}
-
\norm{ d\chi }_{L^{2}(\mathcal{Y}_{0}^{\mathrm{reg}},\gamma_{0})}
\right|
\le
K_{3} |s|^{\frac{1}{2}} \norm{ d\chi }_{L^{2}(\mathcal{Y}_{0}^{\mathrm{reg}},\gamma_{0})}.
\]
\item Let \(\alpha\) be a smooth \((1,1)\)-form on \(\mathcal{Y}\). Then for every
\(\chi \in C_{0}^{\infty}(\mathcal{Y}_{0}^{\mathrm{reg}} \setminus \overline{B}_{\gamma}(\Sing(\mathcal{Y}_{0}), \epsilon(s)))\),
\[
\left|
\int_{\mathcal{Y}_{s}} (F_{s})_{*}\chi \, \alpha_{s}^{n}
-
\int_{\mathcal{Y}_{0}} \chi \, \alpha_{0}^{n}
\right|
\le
K_{4,\alpha} |s|^{\frac{1}{2}} \norm{\chi}_{L^{1}(\mathcal{Y}_{0}^{\mathrm{reg}},\gamma_{0})},
\]
and
\[
\left|
\int_{\mathcal{Y}_{s}} d\bigl((F_{s})_{*}\chi\bigr) \wedge d^{c}\bigl((F_{s})_{*}\chi\bigr) \wedge \alpha_{s}^{n-1}
-
\int_{\mathcal{Y}_{0}} d\chi \wedge d^{c}\chi \wedge \alpha_{0}^{n-1}
\right|
\le
K_{4,\alpha} |s|^{\frac{1}{2}} \norm{d\chi}_{L^{2}(\mathcal{Y}_{0}^{\mathrm{reg}},\gamma_{0})}^{2}.
\]
\end{enumerate}
\end{proposition}

\begin{remark}
    The construction of the retraction using a vector flow is classic,
    as seen in Clemens' retraction \cites{Clemens1969,Clemens1977} and in \cite{Kodaira1986}*{Proof of Th.~2.3}. 
    To the author's knowledge, a quantitative version of this retraction first appeared
    in \cite{Dai-Yoshikawa2025}*{Section~6} 
    for the degeneration of Riemann surfaces.
    This flow construction was also utilized recently
    in \cite{CGP2021}*{Section~2}
    to study complex Monge-Amp\`ere equations.
\end{remark}

We construct a family of diffeomorphisms which sends
the smooth part of the singular fiber to nearby smooth fibers. 
The idea is to pick a \(C^\infty\) complex vector field \(v\) on \(\cY \setminus \mathrm{Crit}(\pi) = \cY \setminus \Sing \cY_0\) satisfying \(\pi_* v = \pa/\pa s\), and use this vector field to flow points on the singular fiber to nearby smooth fibers. 

We begin with a direct Łojasiewicz-type 
estimate, imitating the same estimate in 
\cite{Dai-Yoshikawa2025}*{Lemma~6.2}.

\begin{lemma}
\label[lemma]{lem:dpi-lojasiewicz}
There exists a constant \(c_{0} > 0\) such that
\[
\norm{d\pi(z)}_{\gamma}^{2} \ge c_{0}\, d_{\gamma}\bigl(z,\Sing(\mathcal{Y}_{0})\bigr)^{2n}
\]
for all \(z \in \pi^{-1}(\mathbb{D}_{\rho})\), after shrinking \(\rho > 0\) if necessary. In particular,
for every integer \(\nu \ge n\) there exists \(c_{\nu} > 0\) such that
\[
\norm{d\pi(z)}_{\gamma}^{2} \ge c_{\nu}\, d_{\gamma}\bigl(z,\Sing(\mathcal{Y}_{0})\bigr)^{2\nu}.
\]
\end{lemma}

\begin{proof}
Choose finitely many adapted coordinate charts \(U_{\lambda} \Subset V_{\lambda}\) covering a neighborhood
of \(\mathcal{Y}_{0}\), where \(V_{\lambda}\) carries coordinates
\((z_{0},\dots,z_{n})\) such that
\[
\pi = z_{0} \cdots z_{p}
\]
for some \(0 \le p \le n\), and the components of \(\mathcal{Y}_{0}\) meeting \(V_{\lambda}\) are given by
\(\{ z_{0} = 0 \}, \dots, \{ z_{p} = 0 \}\). On each \(U_{\lambda}\) the metric \(\gamma\) is uniformly
equivalent to the Euclidean metric \(g_{\rm Euc}\), so it suffices to work in these coordinates.

In \(V_{\lambda}\) we have
\[
d\pi = \sum_{j=0}^{p} z_{0} \cdots \widehat{z_{j}} \cdots z_{p}\, dz_{j},
\]
hence
\[
\norm{d\pi(z)}_{g_{\rm Euc}}^{2}
=
\sum_{j=0}^{p} \bigl| z_{0} \cdots \widehat{z_{j}} \cdots z_{p} \bigr|^{2}.
\]
After reordering the coordinates, we may assume
\[
|z_{0}| \le |z_{1}| \le \cdots \le |z_{p}|.
\]
Since \(\{ z_{0} = z_{1} = 0 \} \subset \Sing(\mathcal{Y}_{0})\), we have
\[
d_{\gamma}\bigl(z,\Sing(\mathcal{Y}_{0})\bigr)^{2}
\lesssim
|z_{0}|^{2} + |z_{1}|^{2}.
\]
On the other hand,
\[
\norm{d\pi(z)}_{g_{\rm Euc}}^{2}
\ge
|z_{1} \cdots z_{p}|^{2} + |z_{0} z_{2} \cdots z_{p}|^{2}
\ge
|z_{1}|^{2p} + |z_{0}|^{2p}
\ge
2^{1-p}\bigl(|z_{0}|^{2} + |z_{1}|^{2}\bigr)^{p}.
\]
Therefore
\[
\norm{d\pi(z)}_{\gamma}^{2}
\gtrsim
d_{\gamma}\bigl(z,\Sing(\mathcal{Y}_{0})\bigr)^{2p}
\ge
d_{\gamma}\bigl(z,\Sing(\mathcal{Y}_{0})\bigr)^{2n}
\]
whenever \(z\) stays in a sufficiently small neighborhood of \(\Sing(\mathcal{Y}_{0})\), because \(p \le n\)
and \(d_{\gamma}(z,\Sing(\mathcal{Y}_{0})) \le 1\) there. Away from a fixed neighborhood of
\(\Sing(\mathcal{Y}_{0})\), the function \(\norm{d\pi}_{\gamma}\) has a positive lower bound on
\(\pi^{-1}(\mathbb{D}_{\rho})\) after shrinking \(\rho\), so the same inequality holds globally, possibly
with a smaller constant.

The second assertion follows immediately from the first, since
\(d_{\gamma}(z,\Sing(\mathcal{Y}_{0})) \le 1\) near \(\mathcal{Y}_{0}\) and \(\nu \ge n\).
\end{proof}

\begin{remark}
    Because $\abs{\di \pi}^2$ is locally real analytic, an estimate analogous to the one in \Cref{lem:dpi-lojasiewicz} can be derived from the first Łojasiewicz inequality (cf.\ \cite{malgrange1966ideals}*{p.~62, Theorem~4.1}). It should be noted, however, that the general Łojasiewicz inequality does {not} yield precise information regarding the exponent $\nu$ in \Cref{lem:dpi-lojasiewicz}.
\end{remark}

Then we introduce the vector field used in the construction of \Cref{prop:quantitative-retraction}. We follow
closely the strategy in \cite{Dai-Yoshikawa2025}*{Section~6}.

On \(\mathcal{Y}\setminus \Sing(\mathcal{Y}_{0})\), define the \((1,0)\)-vector field
\[
\Theta:=\frac{(d\pi)^{\sharp_{\gamma}}}{\norm{d\pi}_{\gamma}^{2}} .
\]
Then \(\pi_{*}\Theta=\partial/\partial s\). We define real vector fields \(U,V\) by
\(U-iV:=2\Theta\). If \(s=u+iv\) for $u \coloneqq \mathrm{Re}\,s$
and $v \coloneqq \mathrm{Im}\,s$, then
\[
\pi_{*}U=\frac{\partial}{\partial u},
\qquad
\pi_{*}V=\frac{\partial}{\partial v}.
\]

\begin{lemma}
\label[lemma]{lem:horizontal-vector-field-estimates}
For every \(0<r\leq 1\), there is a constant \(C>0\), independent of \(r\), such that on
\(\pi^{-1}(\mathbb{D}_{\rho})\setminus B_{\gamma}(\Sing(\mathcal{Y}_{0}),r)\),
\[
\abs{U}_{\gamma}+\abs{V}_{\gamma}\leq Cr^{-\nu},
\qquad
\abs{\nabla U}_{\gamma}+\abs{\nabla V}_{\gamma}\leq Cr^{-2\nu}.
\]
Here and below, \(\nabla\) denotes a fixed smooth connection on \(T\mathcal{Y}\), for instance the
Levi-Civita connection of \(\gamma\). Equivalently, in the finitely many adapted charts used above, one may
replace \(\nabla\) by ordinary coordinate differentiation; the resulting norms are uniformly comparable, so
only the constants change.
\end{lemma}

\begin{proof}
Since \(\abs{\Theta}_{\gamma}=\norm{d\pi}_{\gamma}^{-1}\), the first estimate follows from
\Cref{lem:dpi-lojasiewicz}. For the derivative estimate, in the adapted charts the coefficients of
\(\gamma,\gamma^{-1}\), and \(\pi\), together with their derivatives up to order two, are uniformly bounded.
Hence \(\abs{\nabla\Theta}_{\gamma}\leq C\norm{d\pi}_{\gamma}^{-2}\). By
\Cref{lem:dpi-lojasiewicz}, this is bounded by \(Cr^{-2\nu}\) away from
\(B_{\gamma}(\Sing(\mathcal{Y}_{0}),r)\). The estimates for \(U,V\) follow because they are the real and
imaginary parts of \(2\Theta\).
\end{proof}

For \(\theta\in[0,2\pi]\), set \(W^{\theta}:=(\cos\theta)U+(\sin\theta)V\). Then
\[
\pi_{*}W^{\theta}
=
(\cos\theta)\frac{\partial}{\partial u}
+
(\sin\theta)\frac{\partial}{\partial v}
=
e^{i\theta},
\]
where we identify \(T\mathbb{D}\) with \(\mathbb{C}\). Let
\[
M_{r}:=Cr^{-\nu},
\qquad
N_{r}:=Cr^{-2\nu},
\qquad
\delta_{r}:=\frac{r^{2\nu}}{2C}.
\]
Then
\(\delta_{r}\leq \min\{r/M_{r},1/(2N_{r})\}\) for all \(0<r\leq 1\).
Choose \(r_{0}>0\) sufficiently small so that
\[
B_{\gamma}(\Sing(\mathcal{Y}_{0}),2r_{0})\subset \pi^{-1}(\mathbb{D}_{\rho})
\qquad\text{and}\qquad
\delta_{r}<\rho
\quad
(0<r<r_{0}).
\]
In what follows we take \(0<r<r_{0}\). For
\(z\in\mathcal{Y}_{0}^{\mathrm{reg}}\setminus
\overline{B}_{\gamma}(\Sing(\mathcal{Y}_{0}),2r)\), let \(\Phi^{\theta}(\eta,z)\) be the solution of
\begin{equation}
\label{eqn:flow-equation}
\frac{d}{d\eta}\Phi^{\theta}(\eta,z)
=
W^{\theta}_{\Phi^{\theta}(\eta,z)},
\qquad
\Phi^{\theta}(0,z)=z .
\end{equation}
This solution is defined for all \(\abs{\eta}\leq\delta_{r}\). Indeed, let
\((-\tau_{-},\tau_{+})\) be its maximal interval of existence in
\(\pi^{-1}(\mathbb{D}_{\rho})\setminus\Sing(\mathcal{Y}_{0})\). We prove the positive-time direction.

Let \(T_{+}\) be the first time at which the trajectory enters
\(\overline{B}_{\gamma}(\Sing(\mathcal{Y}_{0}),r)\), namely
\[
T_{+}:=\inf\left\{
\eta\in[0,\tau_{+}) \mid
d_{\gamma}\bigl(\Phi^{\theta}(\eta,z),\Sing(\mathcal{Y}_{0})\bigr)\leq r
\right\},
\]
with the convention that \(T_{+}=+\infty\) if the set is empty. If \(T_{+}<+\infty\), then by continuity
\[
\Phi^{\theta}(T_{+},z)\in \partial B_{\gamma}(\Sing(\mathcal{Y}_{0}),r)
\subset \pi^{-1}(\mathbb{D}_{\rho})\setminus\Sing(\mathcal{Y}_{0}),
\]
so the ODE can be continued past \(T_{+}\). Hence \(T_{+}<\tau_{+}\).

We claim that \(T_{+}\geq\delta_{r}\). Otherwise \(T_{+}<\delta_{r}< \infty\), and thus
\(T_{+}<\tau_{+}\). For \(0\leq\eta<T_{+}\), the trajectory stays in
\[
\pi^{-1}(\mathbb{D}_{\rho})
\setminus
B_{\gamma}(\Sing(\mathcal{Y}_{0}),r),
\]
so \(\abs{W^{\theta}}_{\gamma}\leq M_{r}\). Hence
\[
d_{\gamma}\bigl(\Phi^{\theta}(\eta,z),z\bigr)
\leq
\int_{0}^{\eta}
\abs{W^{\theta}_{\Phi^{\theta}(\sigma,z)}}_{\gamma}\,d\sigma
\leq
M_{r}\eta
<
M_{r}\delta_{r}
\leq
r/2 .
\]
Letting \(\eta\to T_{+}\), we get
\(d_{\gamma}(\Phi^{\theta}(T_{+},z),z)\leq r/2\). Since
\(d_{\gamma}(z,\Sing(\mathcal{Y}_{0})) > 2r\), the triangle inequality gives
\[
d_{\gamma}\bigl(\Phi^{\theta}(T_{+},z),\Sing(\mathcal{Y}_{0})\bigr)
\geq \frac{3}{2}r,
\]
contradicting the definition of \(T_{+}\). Thus \(T_{+}\geq\delta_{r}\).

It follows that the trajectory remains in
\(\pi^{-1}(\mathbb{D}_{\rho})\setminus \overline{B}_{\gamma}(\Sing(\mathcal{Y}_{0}),r)\)
for \(0\leq\eta<\min\{\tau_{+},\delta_{r}\}\). If \(\tau_{+}\leq\delta_{r}\), then its image is contained in
the compact subset
\[
\pi^{-1}(\overline{\mathbb{D}}_{\delta_{r}})
\setminus
B_{\gamma}(\Sing(\mathcal{Y}_{0}),r)
\subset
\pi^{-1}(\mathbb{D}_{\rho})\setminus\Sing(\mathcal{Y}_{0}),
\]
where \(W^{\theta}\) is smooth. The ODE continuation theorem therefore extends the solution past
\(\tau_{+}\), contradicting maximality. Hence \(\tau_{+}>\delta_{r}\). The negative-time direction is
identical.

\begin{lemma}
\label[lemma]{lem:flow-basic}
For \(0<r < r_0\), \(\theta\in[0,2\pi]\), \(z\in
\mathcal{Y}_{0}^{\mathrm{reg}}\setminus \overline{B}_{\gamma}(\Sing(\mathcal{Y}_{0}),2r)\), and
\(\abs{\eta}\leq\delta_{r}\), one has
\[
\Phi^{\theta}(\eta,z)\in
\mathcal{Y}_{\eta e^{i\theta}}\setminus \overline{B}_{\gamma}(\Sing(\mathcal{Y}_{0}),r),
\qquad
d_{\gamma}\bigl(\Phi^{\theta}(\eta,z),z\bigr)\leq M_{r}\abs{\eta}.
\]
Moreover, \(\Phi_{\eta}^{\theta}:z\mapsto \Phi^{\theta}(\eta,z)\) is a diffeomorphism from
\(\mathcal{Y}_{0}^{\mathrm{reg}}\setminus \overline{B}_{\gamma}(\Sing(\mathcal{Y}_{0}),2r)\) onto its image in
\(\mathcal{Y}_{\eta e^{i\theta}}\).
\end{lemma}

\begin{proof}
The distance estimate and the exclusion of \(\overline{B}_{\gamma}(\Sing(\mathcal{Y}_{0}),r)\) were proved above.
Moreover,
\[
\frac{d}{d\eta}\pi(\Phi^{\theta}(\eta,z))
=
\pi_{*}W^{\theta}
=
e^{i\theta}.
\]
Since \(\pi(z)=0\), this gives \(\pi(\Phi^{\theta}(\eta,z))=\eta e^{i\theta}\). The diffeomorphism statement
follows from uniqueness of solutions, with inverse obtained by flowing for time \(-\eta\).
\end{proof}

\begin{lemma}[\cite{Dai-Yoshikawa2025}*{Lemma~6.3}]
\label[lemma]{lem:flow-differential}
There exists a constant \(K_{5}>0\) such that for every \(0<r < r_0\), every
\(\theta\in[0,2\pi]\), every \(z\in\mathcal{Y}_{0}^{\mathrm{reg}}\setminus
\overline{B}_{\gamma}(\Sing(\mathcal{Y}_{0}),2r)\), and every \(0\leq\eta\leq\delta_{r}\),
\[
\norm{D\Phi_{\eta}^{\theta}(z)-D\Phi_{0}^{\theta}(z)}
\leq K_{5}N_{r}\eta .
\]
Hence, whenever \(0\leq\eta\leq r^{4\nu}\),
\[
\norm{D\Phi_{\eta}^{\theta}(z)-I}\leq K_{6}\eta^{1/2}.
\]
\end{lemma}

\begin{proof}
Fix a finite atlas by real coordinates near \(\mathcal{Y}_{0}\), and write
\(D\Phi_{\eta}^{\theta}(z)\) for the real Jacobian matrix of \(\Phi_{\eta}^{\theta}\) in these coordinates.
Set \(\Xi^{\theta}(\eta,z):=D\Phi_{\eta}^{\theta}(z)\). Differentiating
\cref{eqn:flow-equation} with respect to the real \(z\)-coordinates gives
\[
\frac{d}{d\eta}\Xi^{\theta}(\eta,z)
=
\bigl(\nabla W^{\theta}\bigr)_{\Phi^{\theta}(\eta,z)}
\cdot\Xi^{\theta}(\eta,z),
\]
where \(\nabla W^{\theta}\) is the real Jacobian matrix of the coordinate expression of \(W^{\theta}\), and
\(\cdot\) is ordinary matrix multiplication. Define
\[
\psi^{\theta}(\eta):=
\norm{\Xi^{\theta}(\eta,z)-\Xi^{\theta}(0,z)}.
\]
Since \(\Phi_{0}^{\theta}\) is the identity map, \(\Xi^{\theta}(0,z)=I\). Thus
\[
\Xi^{\theta}(\eta,z)-\Xi^{\theta}(0,z)
=
\int_{0}^{\eta}
\bigl(\nabla W^{\theta}\bigr)_{\Phi^{\theta}(\sigma,z)}
\cdot\Xi^{\theta}(\sigma,z)\,d\sigma .
\]
Along the trajectory we stay outside \(\overline{B}_{\gamma}(\Sing(\mathcal{Y}_{0}),r)\), so
\(\abs{\nabla W^{\theta}}\leq N_{r}\) by \Cref{lem:horizontal-vector-field-estimates}. 
Hence
\[
\psi^{\theta}(\eta)
\leq
N_{r}\int_{0}^{\eta}\norm{\Xi^{\theta}(\sigma,z)}\,d\sigma
\leq
N_{r}\int_{0}^{\eta}\psi^{\theta}(\sigma)\,d\sigma+N_{r}\eta,
\]
where the last inequality uses
\(\norm{\Xi^{\theta}(\sigma,z)}\leq
\psi^{\theta}(\sigma)+\norm{\Xi^{\theta}(0,z)}\), and the uniform bound for
\(\norm{\Xi^{\theta}(0,z)}\) is absorbed into the constant. By Gronwall's inequality,
\(\psi^{\theta}(\eta)\leq e^{N_{r}\eta}-1\leq 2N_{r}\eta\), since
\(N_{r}\eta\leq N_{r}\delta_{r}\leq 1/2\). This proves the first estimate.

If \(0\leq\eta\leq r^{4\nu}\), then
\[
N_{r}\eta\leq Cr^{-2\nu}\eta\leq Cr^{2\nu}=C\eta^{1/2}.
\]
This gives the second estimate.
\end{proof}

\begin{lemma}[\cite{Dai-Yoshikawa2025}*{Lemma~6.4}]
\label[lemma]{lem:pullback-form-estimate}
Let \(\beta\) be a smooth tensor field on \(\mathcal{Y}\). There exists a constant \(C_{\beta}>0\) such that,
for every \(0<r < r_0\), every \(\theta\in[0,2\pi]\), and every \(0\leq\eta\leq r^{4\nu}\),
\[
\norm{(\Phi_{\eta}^{\theta})^{*}\beta_{\eta e^{i\theta}}-\beta_{0}}_{
L^{\infty}(\mathcal{Y}_{0}^{\mathrm{reg}}\setminus \overline{B}_{\gamma}(\Sing(\mathcal{Y}_{0}),2r))}
\leq C_{\beta}\eta^{1/2}.
\]
\end{lemma}

\begin{proof}
In finitely many adapted charts, the coefficients of \(\beta\) and their first derivatives are uniformly
bounded. By \Cref{lem:flow-basic},
\(d_{\gamma}(\Phi^{\theta}(\eta,z),z)\leq M_{r}\eta\leq C\eta r^{-\nu}\), and by
\Cref{lem:flow-differential}, \(\norm{D\Phi_{\eta}^{\theta}(z)-I}\leq C\eta^{1/2}\). Thus the coefficient
difference between \((\Phi_{\eta}^{\theta})^{*}\beta_{\eta e^{i\theta}}\) and \(\beta_{0}\) is bounded by
\(O(\eta r^{-\nu})+O(\eta^{1/2})\). Since \(\eta\leq r^{4\nu}\), we have
\(\eta r^{-\nu}\leq \eta^{1/2}\). The result follows.
\end{proof}

\begin{proof}[Proof of \Cref{prop:quantitative-retraction}]
If \(s=0\), set \(F_{0}:=\mathrm{id}_{\mathcal{Y}_{0}^{\mathrm{reg}}}\). If \(s\neq0\), write
\(s=\eta e^{i\theta}\) and set \(r:=\abs{s}^{1/(4\nu)}\). For \(\abs{s}\) sufficiently small so that $r < r_0$ and
\(\eta=\abs{s}=r^{4\nu}\leq\delta_{r}\). Define \(F_{s}:=\Phi_{\eta}^{\theta}\) on
\(\mathcal{Y}_{0}^{\mathrm{reg}}\setminus \overline{B}_{\gamma}(\Sing(\mathcal{Y}_{0}),2r)\), which is exactly
\(\mathcal{Y}_{0}^{\mathrm{reg}}\setminus \overline{B}_{\gamma}(\Sing(\mathcal{Y}_{0}),\epsilon(s))\).

Property (1) is immediate. By \Cref{lem:flow-basic},
\[
d_{\gamma}(F_{s}(z),z)\leq M_{r}\abs{s}\leq Cr^{-\nu}\abs{s}
=
C\abs{s}^{3/4},
\]
which proves (2). Property (3) follows from \Cref{lem:pullback-form-estimate}. Applying the same lemma to
the tensor field \(\gamma\), we obtain
\[
\norm{F_{s}^{*}\gamma_{s}-\gamma_{0}}_{L^{\infty}}\leq C\abs{s}^{1/2}.
\]
Since the volume form and the Hodge star operator depend smoothly on the metric,
\[
\norm{\frac{F_{s}^{*}dv_{s}}{dv_{0}}-1}_{L^{\infty}}\leq C\abs{s}^{1/2},
\qquad
\norm{*_{F_{s}^{*}\gamma_{s}}-*_{\gamma_{0}}}_{L^{\infty}}\leq C\abs{s}^{1/2}.
\]
The $L^{\infty}$-norm is $\norm{\cdot}_{L^\infty(\mathcal{Y}_{0}^{\mathrm{reg}}\setminus \overline{B}_{\gamma}(\Sing(\mathcal{Y}_{0}),\epsilon(s)))}$. Here the volume form $dv_{s}$ is $\gamma_s^n$.

For \(\chi,\chi'\in C_{0}^{\infty}(\mathcal{Y}_{0}^{\mathrm{reg}}\setminus
\overline{B}_{\gamma}(\Sing(\mathcal{Y}_{0}),\epsilon(s)))\), change of variables gives
\[
\bigl((F_{s})_{*}\chi,(F_{s})_{*}\chi'\bigr)_{L^{2}(\mathcal{Y}_{s},\gamma_{s})}
=
\int_{\mathcal{Y}_{0}^{\mathrm{reg}}}\chi\,\overline{\chi'}\,F_{s}^{*}dv_{s}.
\]
Hence
\[
\abs{
\bigl((F_{s})_{*}\chi,(F_{s})_{*}\chi'\bigr)_{L^{2}(\mathcal{Y}_{s},\gamma_{s})}
-
(\chi,\chi')_{L^{2}(\mathcal{Y}_{0}^{\mathrm{reg}},\gamma_{0})}
}
\leq
C\abs{s}^{1/2}\norm{\chi}_{L^{2}}\norm{\chi'}_{L^{2}}.
\]
Similarly,
\[
\norm{d((F_{s})_{*}\chi)}_{L^{2}(\mathcal{Y}_{s},\gamma_{s})}^{2}
=
\int_{\mathcal{Y}_{0}^{\mathrm{reg}}}d\chi\wedge *_{F_{s}^{*}\gamma_{s}}\overline{d\chi},
\]
and therefore
\[
\abs{
\norm{d((F_{s})_{*}\chi)}_{L^{2}(\mathcal{Y}_{s},\gamma_{s})}^{2}
-
\norm{d\chi}_{L^{2}(\mathcal{Y}_{0}^{\mathrm{reg}},\gamma_{0})}^{2}
}
\leq
C\abs{s}^{1/2}\norm{d\chi}_{L^{2}(\mathcal{Y}_{0}^{\mathrm{reg}},\gamma_{0})}^{2}.
\]
The corresponding estimate for the unsquared norms follows from
\(\abs{a-b}=\abs{a^{2}-b^{2}}/(a+b)\leq \abs{a^{2}-b^{2}}/b\) when \(b>0\), and is trivial when \(b=0\).
This proves (4).

Finally, for a smooth \((1,1)\)-form \(\alpha\), change of variables and
\Cref{lem:pullback-form-estimate} applied to \(\alpha^{n}\) and \(\alpha^{n-1}\) give
\[
\abs{
\int_{\mathcal{Y}_{s}}(F_{s})_{*}\chi\,\alpha_{s}^{n}
-
\int_{\mathcal{Y}_{0}}\chi\,\alpha_{0}^{n}
}
\leq
C\abs{s}^{1/2}\norm{\chi}_{L^{1}(\mathcal{Y}_{0}^{\mathrm{reg}},\gamma_{0})},
\]
and
\[
\abs{
\int_{\mathcal{Y}_{s}}d((F_{s})_{*}\chi)\wedge d^{c}((F_{s})_{*}\chi)\wedge\alpha_{s}^{n-1}
-
\int_{\mathcal{Y}_{0}}d\chi\wedge d^{c}\chi\wedge\alpha_{0}^{n-1}
}
\leq
C\abs{s}^{1/2}\norm{d\chi}_{L^{2}(\mathcal{Y}_{0}^{\mathrm{reg}},\gamma_{0})}^{2}.
\]
This proves (5), and completes the proof.
\end{proof}

\subsection{Test functions on the central fiber}

We work under the geometric conditions in
\Cref{thm:stable-degeneration-upper-bound}. Then the unique singular fiber
$\cY_0$ is a reduced snc divisor in $\cY$. We shall construct families of
test functions supported on every irreducible component $D_i$ ($1 \leq i \leq a$) of $\cY_0$.

\begin{proposition}
    \label[proposition]{prop:test_functions_on_central}
    There exist positive constants $c_1,c_2>0$ and $\epsilon_0>0$ depending
    on the degeneration $\pi\colon \cY\to\D$ and a \Kahler metric $\gamma$ on
    $\cY$ such that, for every $0<\epsilon<\epsilon_0$, there exists a family
    of smooth functions $\chi^{(i)}_\epsilon$, $1\leq i\leq a$, satisfying:
    \begin{enumerate}
        \item For each $i$, the function $\chi^{(i)}_\epsilon$ depends on
        $\epsilon$ continuously.
        \item For $1\leq i\leq a$, we have
        $\chi^{(i)}_\epsilon\in C^\infty_0(D_i\setminus\Sing(\cY_0))$.
        \item We have $0\leq \chi^{(i)}_\epsilon\leq 1$.
        \item For $y\in D_i$, we have $\chi^{(i)}_\epsilon(y)=0$ if
        $d_\gamma(y,\Sing(\cY_0))\leq c_1\epsilon$, and
        $\chi^{(i)}_\epsilon(y)=1$ if
        $d_\gamma(y,\Sing(\cY_0))\geq c_2\sqrt{\epsilon}$.
        \item We have
        $\norm{\di\chi^{(i)}_\epsilon}_{L^2(\cY_0,\gamma)}^2
        \leq K_4/\log\epsilon^{-1}$, where $K_4>0$ is a uniform constant.
    \end{enumerate}
\end{proposition}

We first define a family of smooth functions on $\R_{>0}$ depending on
$0<\epsilon<1$. Set
\[
u_\epsilon(t)\coloneqq
\frac{\log(t)-\log(\epsilon)}{\log(\sqrt{\epsilon})-\log(\epsilon)}
=
\frac{\log(t)-\log(\epsilon)}{-\frac12\log(\epsilon)}.
\]
Then $u_\epsilon(t)\geq 1$ when $t\geq \sqrt{\epsilon}$,
$u_\epsilon(t)\leq 0$ when $t\leq \epsilon$, and
$u_\epsilon'(t)=-2/(t\log\epsilon)$.

Let $\underline{\eta}\in C^\infty(\R)$ be a standard bump function such that
$0\leq \underline{\eta}\leq 1$, $\underline{\eta}=0$ on $(-\infty,0]$, and
$\underline{\eta}=1$ on $[1,\infty)$. We define
$\vphi_\epsilon(t)\coloneqq(\underline{\eta}\circ u_\epsilon)(t)$. Then
$\vphi_\epsilon=0$ on $(0,\epsilon)$ and $\vphi_\epsilon=1$ on
$(\sqrt{\epsilon},\infty)$. Extending it by zero, we regard
$\vphi_\epsilon$ as a smooth function on $\R$. Moreover,
$\vphi_\epsilon'(t)\neq 0$ only when $\epsilon<t<\sqrt{\epsilon}$, and on
this interval,
\[
\vphi_\epsilon'(t)
=
\underline{\eta}'(u_\epsilon(t))\frac{-2}{t\log(\epsilon)}.
\]
Thus there exists a constant $C_1=C_1(\underline{\eta})$ such that
$\abs{\vphi_\epsilon'(t)}\leq C_1/(t\abs{\log\epsilon})$ for
$\epsilon<t<\sqrt{\epsilon}$, and $\vphi_\epsilon'(t)=0$ elsewhere.

Using the test functions $\vphi_\epsilon$ constructed above, we prove
\Cref{prop:test_functions_on_central} by working on local adapted charts and
gluing them with a partition of unity.

\begin{proof}[Proof of \Cref{prop:test_functions_on_central}]
    We fix one of the irreducible components $D_i$ of $\cY_0$. We assume it
    to be $D_1$ without loss of generality.

    For any point $y\in\cY_0$, we construct a triple of adapted coordinate
    charts around it. Precisely, we have open sets
    $U_y\Subset V_y\Subset W_y$ containing $y$ such that:
    \begin{itemize}
        \item The coordinate chart $\{z_i\}_{i=0}^n$ of $W_y$ is the
        polydisc $\D_3^{n+1}$ of radius $3$. Under this coordinate chart,
        $V_y=\D_2^{n+1}$ and $U_y=\D_1^{n+1}$. The point $y$ is the origin.
        \item The fibration $\pi$ on $W_y$ is given by
        $\pi(z_0,z_1,\cdots,z_n)=z_0z_1\cdots z_p$ for some
        $0\leq p\leq n$, where $z_0,\cdots,z_p$ are defining functions of
        irreducible components of $\cY_0$ intersecting $W_y$.
    \end{itemize}
    Since $\cY_0$ is compact, we can find finitely many such $U_y$ covering
    $\cY_0$. We denote these open sets by $\{U_\alpha\}_\alpha$.

    As $\{V_\alpha\}_\alpha\cup
    (\cY\setminus\cup_\alpha\overline{U}_\alpha)$ constitutes an open
    covering of $\cY$, let
    $\{\eta_\alpha\}_\alpha\cup\{\eta_0\}$ be a partition of unity
    subordinate to this covering. Then $\sum_\alpha\eta_\alpha=1$ on
    $\cup_\alpha U_\alpha$, and each $\eta_\alpha$ is supported in
    $V_\alpha$.

    Each chart system $U_\alpha \Subset V_\alpha
    \Subset W_\alpha$ falls into one of the following cases.
    \begin{itemize}
        \item \textbf{Case 1: When $W_\alpha\cap D_1=\emptyset$.}
        We define
        \begin{equation}
            \label{eqn:case1-contribution}
            \chi^{(1)}_{\epsilon,\alpha}\coloneqq \eta_\alpha\cdot 0=0.
        \end{equation}
        Then $\chi^{(1)}_{\epsilon,\alpha}\mid_{\cY_0}=0$ is a smooth
        function on $\cY_0^{\rm reg}$.

        \item \textbf{Case 2: When $W_\alpha\cap D_1\neq\emptyset$ and
        $W_\alpha\cap\Sing(\cY_0)=\emptyset$.}
        We define
        \begin{equation}
            \label{eqn:case2-contribution}
            \chi^{(1)}_{\epsilon,\alpha}\coloneqq \eta_\alpha\cdot 1
            =\eta_\alpha.
        \end{equation}
        Then $\chi^{(1)}_{\epsilon,\alpha}\mid_{\cY_0}$ is a smooth
        function on $\cY_0^{\rm reg}$, since $\cY_0$ is smooth on
        $V_\alpha$.

        \item \textbf{Case 3: When $W_\alpha\cap D_1\neq\emptyset$ and
        $W_\alpha\cap\Sing(\cY_0)\neq\emptyset$.}
        After reindexing, we assume that
        $\{z_0=0\}\cap W_\alpha=D_1\cap W_\alpha$ and that
        $\pi(z)=z_0z_1\cdots z_p$ on $W_\alpha$ for some $1\leq p\leq n$.
        Then
        $\Sing(\cY_0)\cap W_\alpha
        \subset(\{z_1=0\}\cup\cdots\cup\{z_p=0\})\cap W_\alpha$ on
        $D_1$. We define
        \begin{equation}
            \label{eqn:case3-contribution}
            \chi^{(1)}_{\epsilon,\alpha}
            \coloneqq
            \eta_\alpha\prod_{i=1}^p\vphi_\epsilon(\abs{z_i}).
        \end{equation}
        Then $\chi^{(1)}_{\epsilon,\alpha}$ is a smooth function supported
        in $V_\alpha$. By construction of $\vphi_\epsilon$, it is zero if
        $\abs{z_i}\leq\epsilon$ for some $1\leq i\leq p$. Hence it is zero
        around $\Sing(\cY_0)\cap W_\alpha$, and its restriction to
        $\cY_0$ is smooth on $\cY_0^{\rm reg}$.
    \end{itemize}
    Finally, we define
    \begin{equation}
        \label{eqn:final-test-by-adding-all}
        \chi^{(1)}_\epsilon
        \coloneqq
        \sum_\alpha\chi^{(1)}_{\epsilon,\alpha}\mid_{\cY_0}.
    \end{equation}
    We now verify the five conditions in
    \Cref{prop:test_functions_on_central}.

    Since $\vphi_\epsilon$ depends on $\epsilon$ continuously, so does
    $\chi^{(1)}_\epsilon$. \textbf{This verifies (1).}

    \textbf{We verify that $\chi^{(1)}_\epsilon$ is supported on $D_1$.}
    If $y\in\cY_0\setminus D_1$, then every local contribution vanishes.
    In \textbf{Case 1} this is immediate. In \textbf{Case 2}, the set
    $V_\alpha\cap\cY_0$ lies in the smooth component $D_1$, so
    $\eta_\alpha(y)=0$. In \textbf{Case 3}, since $y\in\cY_0\setminus D_1$,
    we have $z_i(y)=0$ for some $1\leq i\leq p$; thus the product in
    \cref{eqn:case3-contribution} vanishes. Therefore
    $\chi^{(1)}_\epsilon=0$ on $\cY_0\setminus D_1$. From the construction,
    it is also zero around $\Sing(\cY_0)$. Hence
    $\chi^{(1)}_\epsilon\in C^\infty_0(D_1\setminus\Sing(\cY_0))$.
    \textbf{This verifies (2).}

    By construction, every local factor lies between $0$ and $1$, and
    $\sum_\alpha\eta_\alpha=1$ on $\cY_0$. Hence
    $0\leq\chi^{(1)}_\epsilon\leq 1$ on $\cY_0$.
    \textbf{This verifies (3).}

    We now determine the constants for the distance property. Let
    $0<L_1<1<L_2$ be constants such that
    $L_1^2\omega_{\rm Euc}\leq\gamma\mid_{W_\alpha}\leq
    L_2^2\omega_{\rm Euc}$ on every $W_\alpha$, where
    $\omega_{\rm Euc}=\sum_{i=0}^n\sqrt{-1}\di z_i\wedge\di\bar z_i$.
    We set $c_1=L_1$ and $c_2=L_2$.

    We choose $0<r_0<e^{-1}$ as in
    \Cref{lem:existence-r0-for-central-cutoff}. Thus, if
    $y\in D_1$ and $d_\gamma(y,\Sing(\cY_0))<r_0$, then every
    $V_\alpha$ with $y\in V_\alpha$ is of \textbf{Case 3}. Moreover,
    for such a chart,
    \begin{equation}
        \label{eqn:distance-mu-comparison}
        \begin{aligned}
        &L_1\min_{1\leq i\leq p}\abs{y_i}
        \leq d_\gamma(y,\Sing(\cY_0))
        \quad\text{if }d_\gamma(y,\Sing(\cY_0))<r_0,\\
        &d_\gamma(y,\Sing(\cY_0))
        \leq L_2\min_{1\leq i\leq p}\abs{y_i}.
        \end{aligned}
    \end{equation}
    We take $\epsilon_0=(r_0/c_2)^2$. Since $r_0<e^{-1}$ and $c_2>1$, we
    have $\epsilon_0<e^{-2}$. Also, for every $0<\epsilon<\epsilon_0$, we
    have $c_1\epsilon<r_0$ and $c_2\sqrt{\epsilon}<r_0$.

    \textbf{We verify the distance property in (4).}
    If $y\in D_1$ and $d_\gamma(y,\Sing(\cY_0))\leq c_1\epsilon$, then
    $d_\gamma(y,\Sing(\cY_0))<r_0$. Hence every nonzero local contribution
    comes from \textbf{Case 3}. By \cref{eqn:distance-mu-comparison},
    $L_1\min_i\abs{y_i}\leq d_\gamma(y,\Sing(\cY_0))\leq c_1\epsilon$.
    Since $c_1=L_1$, we get $\min_i\abs{y_i}\leq\epsilon$, and so
    $\chi^{(1)}_{\epsilon,\alpha}(y)=0$ for every such $\alpha$.
    Therefore $\chi^{(1)}_\epsilon(y)=0$.

    Conversely, if $y\in D_1$ and
    $d_\gamma(y,\Sing(\cY_0))\geq c_2\sqrt{\epsilon}$, then \textbf{Case 1}
    does not contribute. In \textbf{Case 2}, the local factor is identically
    $1$. In \textbf{Case 3}, \cref{eqn:distance-mu-comparison} gives
    $L_2\min_i\abs{y_i}\geq d_\gamma(y,\Sing(\cY_0))\geq
    c_2\sqrt{\epsilon}$. Since $c_2=L_2$, we get
    $\min_i\abs{y_i}\geq\sqrt{\epsilon}$. Thus
    $\chi^{(1)}_{\epsilon,\alpha}(y)=\eta_\alpha(y)$ in every contributing
    chart, and summing over $\alpha$ gives $\chi^{(1)}_\epsilon(y)=1$.
    \textbf{This verifies (4).}

    \textbf{We verify the gradient estimate in (5).}
    By the distance property, $\di\chi^{(1)}_\epsilon$ is supported in
    $A_\epsilon\coloneqq D_1\cap
    \{c_1\epsilon\leq d_\gamma(y,\Sing(\cY_0))\leq c_2\sqrt{\epsilon}\}$.
    Since $c_2\sqrt{\epsilon}<r_0$, only charts of \textbf{Case 3} occur on
    $A_\epsilon$. Hence, on $A_\epsilon$,
    \[
    \di\chi^{(1)}_\epsilon
    =
    \sum_{\substack{\alpha\text{ in}\\\textbf{Case 3}}}
    \di\left(\eta_\alpha\prod_{i=1}^p\vphi_\epsilon(\abs{z_i})\right).
    \]
    Write $P_{\epsilon,\alpha}\coloneqq
    \prod_{i=1}^p\vphi_\epsilon(\abs{z_i})$. Then
    $\di(\eta_\alpha P_{\epsilon,\alpha})
    =P_{\epsilon,\alpha}\di\eta_\alpha+\eta_\alpha\di P_{\epsilon,\alpha}$.
    By finite multiplicity of the covering,
    \[
    \begin{aligned}
        \norm{\di\chi^{(1)}_\epsilon}_{L^2(\cY_0,\gamma)}^2
        &\leq
        C\int_{A_\epsilon}
        \abs{\sum_{\alpha}P_{\epsilon,\alpha}\di\eta_\alpha}_\gamma^2
        \,dV_\gamma  \\
        &\quad+
        C\sum_{\substack{\alpha\text{ in}\\\textbf{Case 3}}}
        \int_{D_1\cap V_\alpha}\abs{\di P_{\epsilon,\alpha}}_\gamma^2
        \,dV_\gamma .
    \end{aligned}
    \]
    The first term is bounded by $C\epsilon$, because $A_\epsilon$ is
    contained in the $c_2\sqrt{\epsilon}$-neighborhood of
    $\Sing(\cY_0)\cap D_1$, this neighborhood has volume $O(\epsilon)$ in
    $D_1$, and the functions $\eta_\alpha$ are fixed.

    For the second term, the estimate for $\vphi_\epsilon'$ gives, on each
    \textbf{Case 3} chart,
    \[
        \abs{\di P_{\epsilon,\alpha}}_\gamma^2
        \leq
        \frac{C}{\abs{\log\epsilon}^2}
        \sum_{i=1}^p
        \frac{\mathbf{1}_{\{\epsilon<\abs{z_i}<\sqrt{\epsilon}\}}}
        {\abs{z_i}^2}.
    \]
    Using the equivalence between $\gamma$ and the Euclidean metric, and
    integrating over the remaining bounded coordinates, we obtain
    \[
        \int_{D_1\cap V_\alpha}
        \abs{\di P_{\epsilon,\alpha}}_\gamma^2\,dV_\gamma
        \leq
        \frac{C}{\abs{\log\epsilon}^2}
        \sum_{i=1}^p\int_\epsilon^{\sqrt{\epsilon}}\frac{dr}{r}
        \leq
        \frac{C}{\log\epsilon^{-1}}.
    \]
    Since the number of charts is finite and $\epsilon_0<e^{-2}$, the
    harmless $C\epsilon$ term is also bounded by $C/\log\epsilon^{-1}$.
    Therefore
    \[
        \norm{\di\chi^{(1)}_\epsilon}_{L^2(\cY_0,\gamma)}^2
        \leq
        \frac{K_4}{\log\epsilon^{-1}},
    \]
    where $K_4$ depends only on the finite adapted covering, the partition
    of unity, the fixed bump function $\underline{\eta}$, and the metric
    comparison constants for $\gamma$. \textbf{This verifies (5).}

    The construction for the other irreducible components $D_i$ is the same
    after reindexing the adapted coordinates. This completes the proof.
\end{proof}

Lastly, we prove the following lemma used above.
\begin{lemma}
    \label[lemma]{lem:existence-r0-for-central-cutoff}
    With the finite adapted covering fixed above, there exists
    $0<r_0<e^{-1}$ such that the following holds. If $y\in D_1$ and
    $d_\gamma(y,\Sing(\cY_0))<r_0$, then every $V_\alpha$ with
    $y\in V_\alpha$ is of \textbf{Case 3}. Moreover, for every such chart,
    the distance comparison \cref{eqn:distance-mu-comparison} holds.
\end{lemma}

\begin{proof}
    Since charts of \textbf{Case 1} do not meet $D_1$, it suffices to
    exclude charts of \textbf{Case 2}.

    Since $V_\alpha\Subset W_\alpha$ and the family of charts is finite, the
    number
    \[
        \delta_{\rm bd}\coloneqq
        \min_\alpha d_\gamma(\overline{V}_\alpha,\cY\setminus W_\alpha)
    \]
    is positive. Choose
    \(
        0<r_0<\min\{e^{-1},\delta_{\rm bd}\}.
    \)

    If $y\in D_1$ and $d_\gamma(y,\Sing(\cY_0))<r_0$, then no chart of
    \textbf{Case 2} can contain $y$. Indeed, if $y\in V_\alpha$ for some
    $\alpha$ in \textbf{Case 2}, then
    $\Sing(\cY_0)\subset \cY\setminus W_\alpha$, and therefore
    \[
        d_\gamma(y,\Sing(\cY_0))
        \ge d_\gamma(y,\cY\setminus W_\alpha)
        \ge d_\gamma(\overline{V}_\alpha,\cY\setminus W_\alpha)
        \ge \delta_{\rm bd}
        > r_0,
    \]
    a contradiction. Hence every $V_\alpha$ with $y\in V_\alpha$ is of
    \textbf{Case 3}.

    It remains to prove the distance comparison in
    \cref{eqn:distance-mu-comparison}. Suppose $y\in D_1\cap V_\alpha$ and
    $V_\alpha$ is of \textbf{Case 3}. Since
    $d_\gamma(y,\Sing(\cY_0))<r_0\leq\delta_{\rm bd}$, every closest point of
    $\Sing(\cY_0)$ to $y$ lies in $W_\alpha$. Thus the distance is computed
    inside $W_\alpha$. In this chart, $D_1=\{z_0=0\}$ and
    $\Sing(\cY_0)\cap D_1\cap W_\alpha
    =\cup_{i=1}^p\{z_0=z_i=0\}$. The other local strata of
    $\Sing(\cY_0)$ are no closer to $y=(0,y_1,\cdots,y_n)$. Therefore the
    Euclidean distance from $y$ to the local singular locus is
    $\min_{1\leq i\leq p}\abs{y_i}$.

    Since $L_1^2\omega_{\rm Euc}\leq\gamma\leq
    L_2^2\omega_{\rm Euc}$ on $W_\alpha$, the metric distance satisfies
    $L_1\min_i\abs{y_i}\leq d_\gamma(y,\Sing(\cY_0))$ whenever
    $d_\gamma(y,\Sing(\cY_0))<r_0$. Conversely, joining $y$ to the local
    singular locus along a coordinate line gives
    $d_\gamma(y,\Sing(\cY_0))\leq L_2\min_i\abs{y_i}$. This proves the
    lemma.
\end{proof}

\subsection{Proof of 
\mathToString{\Cref{thm:stable-degeneration-upper-bound}}}
We prove \Cref{thm:stable-degeneration-upper-bound} by constructing test functions on $\cY_s$ by
flowing test functions on the central fiber
(\Cref{prop:test_functions_on_central})
to nearby $\cY_s$ using \Cref{prop:quantitative-retraction}.

We work under the notations in \Cref{thm:stable-degeneration-upper-bound} and $\beta$
is the semi-positive $(1, 1)$-form therein.
We write  
\[
\mathcal{Y}_{0} = \sum_{i=1}^{N} D_{i} + \sum_{i=N+1}^{a} D_{i},
\]
where $D_1, \cdots, D_N$ are exactly the irreducible components
of $\cY_0$ such that $\int_{D_i} \beta^n > 0$.

For each $1 \leq i \leq N$, under the notations in \Cref{prop:test_functions_on_central}, we define
\[
\chi^{(i)}_s \coloneqq \chi^{(i)}_{2\epsilon(s)/c_1},
\]
where $\epsilon(s) = 2 \abs{s}^{\frac{1}{4\nu}}$. Since 
$\chi^{(i)}_s \in C_{0}^{\infty}(D_i \setminus \overline{B}_{\gamma}(\Sing(\mathcal{Y}_{0}), \epsilon(s)))$ by \Cref{prop:test_functions_on_central}, we
define 
\[
\tilde{u}_{i, s} \coloneqq (F_s)_* \chi^{(i)}_s \in C^\infty(\cY_s),
\]
where $F_s$ is the diffeomorphism introduced in \Cref{prop:quantitative-retraction}. We shall estimate the $L^2$-norms
of $\tilde{u}_{i, s}$ and its derivative on $(\cY_s, \beta_s)$.

Using \Cref{prop:quantitative-retraction}(5), we have 
\[
\int_{\cY_s} \tilde{u}_{i, s}^2 \beta_s^n = \int_{\cY_0} (\chi^{(i)}_s)^2 \beta_0^n 
+ O(\abs{s}^{\frac{1}{2}}) = \int_{D_i} (\chi^{(i)}_s)^2 \beta_0^n 
+ O(\abs{s}^{\frac{1}{2}}).
\]

From \Cref{prop:test_functions_on_central}(4), $\chi^{(i)}_s$
is $1$ on $D_i$ outside a neighborhood of $\Sing(\cY_0)$ of radius
$c_2 \sqrt{2\epsilon(s)/c_1}$, which has volume $O(\epsilon(s))$
in $D_i$, we have 
\begin{align*}
    \int_{D_i} (\chi^{(i)}_s)^2 \beta_0^n 
    &= \int_{D_i} \beta_0^n - \int_{D_i \cap d_\gamma(y,\Sing(\cY_0))\leq c_2\sqrt{2\epsilon(s)/c_1}} (1 - (\chi^{(i)}_s)^2) \beta_0^n
    \\
    &= \int_{D_i} \beta_0^n + O(\epsilon(s)).
\end{align*}
Thus $\norm{\tilde{u}_{i, s}}^2_{L^2(\cY_s, \beta_s)} = \int_{D_i} \beta_0^n + O(\abs{s}^{\frac{1}{4\nu}})$.

For the derivative, using \Cref{prop:quantitative-retraction}(5),
\[
\int_{\cY_s} \norm{\di \tilde{u}_{i, s}}_{\beta_s}^2 \beta_s^n
= n\int_{\cY_0} \di \chi^{(i)}_s \wedge \di^c \chi^{(i)}_s \wedge\beta_0^{n - 1} + O(\abs{s}^{\frac{1}{2}}),
\]
we recall that $\beta_s$ is a \Kahler form on the regular fiber 
$\cY_s$.

Using \Cref{prop:test_functions_on_central}(5) and choosing 
$\gamma$ to be a background \Kahler that dominates $\beta$, we have 
\[
n\int_{\cY_0} \di \chi^{(i)}_s \wedge \di^c \chi^{(i)}_s \wedge\beta_0^{n - 1} = O(1/\log(\epsilon(s)^{-1}))
= O\left({\log^{-1}(\abs{s}^{-1})}\right).
\]
Thus $\norm{\di \tilde{u}_{i, s}}^2_{L^2(\cY_s, \beta_s)} = O\left({\log^{-1}(\abs{s}^{-1})}\right)$.

For each $1 \leq i \leq N$, we have $\int_{D_i} \beta_0^n > 0$.
So $\norm{\tilde{u}_{i, s}}^2_{L^2(\cY_s, \beta_s)} = \int_{D_i} \beta_0^n + O(\abs{s}^{\frac{1}{4\nu}}) > 0$ for 
$s$ small enough.
Then we define 
\[
{u}_{i, s}= \frac{\tilde{u}_{i, s}}{\norm{\tilde{u}_{i, s}}_{L^2(\cY_s, \beta_s)}}.
\]
Note that ${u}_{i, s}$ is supported on $F_s(D_i \setminus \overline{B}_{\gamma}(\Sing(\mathcal{Y}_{0}), \epsilon(s)))$. These supports are mutually disjoint.
Thus $\{u_{i, s}\}_{i=1}^N$
is an orthonormal system in $L^2(\cY_s, \beta_s)$.

By the min-max principle, for $\lambda_k(s)$ the $k$-th
non-zero eigenvalue of the Laplacian on $(\cY_s, \beta_s)$,
we have 
\[
\lambda_k(s) = \min_{\substack{V \subset C^\infty(\cY_s)\\
\dim V = k + 1}} \max_{\substack{u \in V, \norm{u}= 1}} \norm{\di u}_{L^2(\cY_s, \beta_s)}^2.
\]

For $1 \leq k \leq N - 1$, we let the $k+1$-dimensional vector
space $V$ appearing in the min-max formula be the space 
spanned by ${u}_{1, s}, \cdots, u_{k+1, s}$. Note that 
\[
\norm{\di {u}_{i, s}}^2_{L^2(\cY_s, \beta_s)}
= \frac{O\left({\log^{-1}(\abs{s}^{-1})}\right)}{\int_{D_i} \beta_0^n + O(\abs{s}^{\frac{1}{4\nu}})} = O\left({\log^{-1}(\abs{s}^{-1})}\right), \quad 
1 \leq i \leq N.
\]
The supports of ${u}_{1, s}, \cdots, u_{k+1, s}$ are mutually disjoint.

We have 
\[
\lambda_k(s) = O\left({\log^{-1}(\abs{s}^{-1})}\right)
\]
for $1 \leq k \leq N -1$. This completes the proof of \Cref{thm:stable-degeneration-upper-bound}.

\section{Applications and Discussions}
\subsection{Applications in geometric analysis}
\label[subsection]{sec:applications-subsection}
Let $\pi \colon (X, \omega_X) \to \D_s$ be a degeneration of \Kahler manifolds of complex dimension $n$. Assume $X_0$ is the unique singular fiber. Let $\omega_s \eqqcolon \omega_X|_{X_s}$ denote the restricted metric on the smooth fibers $X_s$.

When $X_0$ is reduced and irreducible, 
the first non-zero eigenvalue of the Laplacian 
on $(X_s, \omega_s)$ is uniformly bounded away from $0$ as $s \to 0$.
Consequently, there exists a uniform constant $C > 0$ such that the 
following $L^2$-Poincaré inequality holds for all 
$s \in \D^\circ_{\frac{1}{2}}$ and all $f \in C^{\infty}(X_s)$:
\begin{equation}
    \label{eqn:uniform_poincare}
    \norm{f - \frac{\int_{X_s} f \omega_s^n}{\int_{X_s} \omega_s^n}}_{L^2(X_s, \omega_s)} \leq C \norm{\di f}_{L^2(X_s, \omega_s)}.
\end{equation}

This uniform inequality is widely utilized in geometric settings 
involving the degeneration of complex manifolds (see, for example, 
\cite{Ruan-Zhang2011}*{Proposition~3.2},  
\cite{CGP2021}*{Appendix}, \cite{DNGG2022}*
{Proposition~3.10}, \cite{Pan2022}*{Proposition~2.2}
and \cite{Filip-Tosatti21}*{Theorem 3.2.4}).

When $X_0$ is not irreducible, it is well known that the uniform Poincaré 
constant in \Cref{eqn:uniform_poincare} blows up as $s \to 0$. 
However, \Cref{thm:main-thm-combined} ensures that this rate of 
blow-up is strictly controlled. This control ultimately yields the 
following lower bound for the Green function $G_s(x, y)$ on the 
fibers $(X_s, \omega_s)$, which remains valid even for general 
singular fibers:
\begin{proposition}
    \label[proposition]{prop:green-function-general-sing}
    Let $G_s(x, y)$ be the Green function of the \Kahler manifold
    $(X_s, \omega_s)$ for $s \neq 0$. Then there exists a constant $C > 0$ such that 
    \[
    G_s(x, y) \geq -C \abs{\log \abs{s}}
    \]
    for all 
    $s \in \D^{\circ}_{\frac{1}{2}}$.
\end{proposition}
\begin{proof}
    Let $\{\Phi_{i}(s, z)\}_{i \geq 1}$ be eigenfunctions 
    of $\Delta_{\omega_s}$ with non-zero eigenvalues
    and normalized $L^2$-norms.
    Let $k_0$ be the number of {small eigenvalues}. 
    Then we have 
    \[
    G_s(x, y) = \underbrace{\sum_{k = 1}^{k_0}\frac{1}{\lambda_k(s)} \Phi_k(s,x) \Phi_k(s,y)}_{\eqqcolon G_{s, \mathrm{low}}(x, y)}
    + \underbrace{\sum_{k = k_0 + 1}^{\infty}\frac{1}{\lambda_k(s)} \Phi_k(s,x) \Phi_k(s,y)}_{\eqqcolon G_{s, \mathrm{high}}(x, y)}.
    \]

Since $\lambda_{k_0 + 1}(s)$ is not a {small eigenvalue}, it 
    is uniformly bounded away from $0$. Consequently, a classical 
    argument by Cheng and Li \cite{Cheng_Li1981} ensures that 
    $G_{s, \mathrm{high}}(x, y) \geq -C_1$ for some uniform constant 
    $C_1 > 0$. (See also 
    \cite{Cao2026}*{Proposition~3.6}, noting that Theorem~3.7 used therein can be replaced by the higher-dimensional version in \cite{CKS87}*{Theorem~2.1}.)

    For the remaining terms, \Cref{thm:main-thm-combined} provides control over the 
    small eigenvalues. Furthermore, uniform Sobolev 
    inequalities combined with Moser iteration bound 
    the $L^\infty$-norms of the corresponding small 
    eigenfunctions. Indeed, from \Cref{eqn:eigen-infinity-norm}, we have 
    \[
    \norm{\Phi_k}_{L^\infty(X_s)}\leq
    \exp\left(C_{\rm Sob}
    \frac{\sqrt{\nu \lambda}}{\sqrt{\nu} - 1}\right).
    \]
    Here $\lambda > 0$ is a uniform bound 
    for small eigenvalues, i.e., 
    $0 < \lambda_1(s) \leq \cdots \leq \lambda_{k_0}(s) \leq \lambda < \infty$ for $s \in \D_{1/2}$, and $\nu = \nu(n)$ is the exponent in 
    the Sobolev inequality.
    
    Together, these bounds yield 
    \[
    G_{s, \mathrm{low}}(x, y) \geq -C_2 \abs{\log \abs{s}}
    \]
    for a uniform constant $C_2 > 0$.

    Summing these estimates yields the desired bound $G_s(x, y) \geq -C \abs{\log \abs{s}}$.
\end{proof}

The estimate of the fiberwise Green functions leads to the 
following two applications.

\subsubsection{Estimates of families of plurisubharmonic functions}

\begin{proposition}
    Let $\theta$ be a smooth $(1, 1)$-form on $X$. Then there is a 
    uniform constant $C_1 > 0$ such that for $s \in \D^{\circ}_{\frac{1}{2}}$, we have 
    \[
    C_1 \log\abs{s} \leq \frac{1}{\vol(\omega_s)}\int_{X_s} \vphi \omega_s^n  - \sup_{X_s} \vphi \leq 0
    \]
    for all $\vphi \in \PSH(X_s, \theta_s)$. Here $\vol(\omega_s) = \int_{X_s}\omega_s^n$ is a positive constant 
    determined by the cohomology class of $\omega_X$.
\end{proposition}
\begin{proof}
    Since $\pi^{-1}(\overline{\D_{3/4}})$ is compact, we take 
    a large constant $L > 0$ such that $\theta < L \omega_X$
    on $\pi^{-1}(\overline{\D_{3/4}})$. Then for $s \in \D^{\circ}_{1/2}$, we have $\PSH(X_s, \theta_s) \subset 
    \PSH(X_s, L\omega_s)$. So without loss of generality, 
    we assume $\theta = \omega_s$.

    Since $\omega_s + \ddc \vphi \geq 0$ for $\vphi \in \PSH(X_s, \omega_s)$, we have 
    \begin{equation}
        \label{eqn:Laplacian-of-psh}
        -\Delta_{\omega_s} \vphi = \tr_{\omega_s} (\ddc \vphi)
    = \frac{n \ddc \vphi \wedge \omega_s^{n-1}}{\omega_s^n}
    \geq -n.
    \end{equation}
    Note that we use the positive definite Laplacian. So 
    \[
    \frac{1}{\vol(\omega_s)}\int_{X_s} \vphi \omega_s^n
    - \vphi(z) = \int_{z' \in X_s} G_{\omega_s}(z, z')
    (-\Delta_{\omega_s} \vphi)(z') \omega_s^n(z').
    \]

    By \Cref{prop:green-function-general-sing}, we have 
    \begin{align*}
        &\int_{z' \in X_s} (G_{\omega_s}(z, z') + C \abs{\log \abs{s}})
        (-\Delta_{\omega_s} \vphi)(z') \omega_s^n(z') \\
        &\geq \int_{z' \in X_s} (G_{\omega_s}(z, z') + C \abs{\log \abs{s}})
        (-n) \omega_s^n(z')\\
        &(\text{Using }G_{\omega_s}(z, z') + C \abs{\log \abs{s}} \geq 0 \text{ and } \cref{eqn:Laplacian-of-psh})\\
        &\geq  \int_{z' \in X_s}C \abs{\log \abs{s}}
        (-n) \omega_s^n(z') = -nC \vol(\omega_s) \abs{\log \abs{s}}\\
        &(\text{Using }\int_{z' \in X_s}G_{\omega_s}(z, z') \omega_s^n(z') = 0 ).
    \end{align*}
    Since $\int_{z' \in X_s} C \abs{\log \abs{s}}
    (-\Delta_{\omega_s} \vphi) \omega_s^n = 0$, we have 
    \[
    \frac{1}{\vol(\omega_s)}\int_{X_s} \vphi \omega_s^n
    - \vphi(z) 
    = \int_{z' \in X_s} (G_{\omega_s}(z, z') + C \abs{\log \abs{s}})
    (-\Delta_{\omega_s} \vphi)(z') \omega_s^n(z')
    \geq C_1 \log \abs{s}.
    \]
    Here $C_1 = n \vol(\omega_s) C$.

    So we have 
    \[
    C_1 \log\abs{s} \leq \frac{1}{\vol(\omega_s)}\int_{X_s} \vphi \omega_s^n  - \sup_{X_s} \vphi \leq 0.
    \]
\end{proof}

\begin{remark}
    When the singular fiber $X_0$ is reduced and irreducible,
    the $\log\abs{s}$ factor can be removed due to 
    \cite{DNGG2022}*{Conjecture~3.1} and \cite{Ou2022}*{Corollary~4.8}.
\end{remark}

\begin{remark}
    The $\log\abs{s}$ factor is optimal. Indeed, in \cite{DNGG2022}*{Example~3.5}, they construct a family of plurisubharmonic
    functions $\vphi_s \in \PSH(X_s, \omega_s)$ so that $\sup_{X_s} \vphi_s = 0$ and 
    \[
    \int_{X_s} \vphi_s \omega_s^n \leq C \log\abs{s}
    \]
    for a constant $C > 0$.
\end{remark}

\subsubsection{Estimates of families of Poisson equations}

\begin{proposition}
    \label[proposition]{prop:poisson-degenerating}
    Let $g$ be a continuous function on $X$ 
    and suppose it satisfies the integrability condition on all smooth fibers $X_s$ ($s \neq 0$),
    \[
    \int_{X_s} g \omega_s^n = 0.
    \]
    Then for $s \in \D^{\circ}_{\frac{1}{2}}$, let $\vphi_s$ be the unique 
    solution to the Poisson equation on $X_s$, i.e.,
    \[
    \Delta_{\omega_s} \vphi_s = g\mid_{X_s}, \quad
    \int_{X_s} \vphi_s \omega_s^n = 0.
    \]
    We have a positive constant $C_2 > 0$ independent of 
    $g$ such that 
    \[
    \norm{\vphi_s}_{L^\infty(X_s)} \leq C_2 \abs{\log \abs{s}}
    \norm{g}_{L^\infty(X_s)}
    \]
    for all $s \in \D^{\circ}_{\frac{1}{2}}$.
\end{proposition}
\begin{proof}
    For $s \in \D^{\circ}_{\frac{1}{2}}$, we have 
    \[
    \vphi_s(z) = \int_{z' \in X_s} G_s(z, z') g(z')
    \omega_s^n(z').
    \]
    Using $G_{\omega_s}(z, z') + C \abs{\log \abs{s}} \geq 0$ in \Cref{prop:green-function-general-sing} and 
    $g + \norm{g}_{L^\infty} \geq 0$, we obtain that 
    \begin{align*}
        \int_{z' \in X_s} G_s(z, z') g(z')
    \omega_s^n(z') &= \int_{z' \in X_s} (G_s(z, z') + C \abs{\log \abs{s}}) g(z')
    \omega_s^n(z')\\
    &\geq \int_{z' \in X_s} (G_s(z, z') + C \abs{\log \abs{s}}) (-\norm{g}_{L^\infty})
    \omega_s^n(z')\\
    &= -C \norm{g}_{L^\infty} \vol(\omega_s)\abs{\log \abs{s}}.
    \end{align*}
    Here $\norm{g}_{L^\infty} \coloneqq 
    \norm{g}_{L^\infty(X_s)}$. For the reverse direction, we consider $-g$. This completes the 
    proof.
\end{proof}

\begin{remark}
    In the study of holomorphic dynamics, we need to solve fiberwise 
    Poisson equations for degenerating families of 
    \Kahler manifolds, e.g. in \cite{Filip-Tosatti21}*{Theorem 3.2.4}
    and in \cite{Cao2026}.
\end{remark}

\begin{remark}
    When we assume the degeneration $\pi \colon X \to \D$
    to be semistable and the continuous function $g$ also satisfies
    the integrability condition on central fibers, 
    \[
    \int_{D} g \omega^n = 0,
    \]
    where $D$ runs over all irreducible components of the singular 
    fiber, then we have better estimates on $\vphi_s$, the solution
    to the fiberwise Poisson equations in \Cref{prop:poisson-degenerating}. Indeed, following the same argument
    in \cite{Cao2026}*{Theorem 3.5}, we have 
    \[
    \frac{\norm{\vphi_s}_{L^\infty(X_s)}}{\abs{\log \abs{s}}^{1/2}} 
    \to 0, \quad s \to 0.
    \]
\end{remark}

\subsection{Small eigenvalues of degenerating \Kahler-Einstein manifolds}
\label[subsection]{sec:KE-small-eigen-examples-subsection}

In this subsection, we demonstrate that small eigenvalues also 
appear for degenerating families of compact \Kahler manifolds with 
canonical metrics.

\begin{example}
    \label[example]{example:calabi-yau}
    Let $\pi \colon X \to \D$ be a degenerating family of Calabi--Yau manifolds, 
    polarized by a relatively ample line bundle $L$. We equip each smooth fiber $X_s$ 
    ($s \neq 0$) with the unique Ricci-flat Calabi--Yau metric $\omega^{\rm CY}_s$ 
    representing the K\"ahler class $c_1(L)|_{X_s}$. Let $\lambda_k^{\rm CY}(s)$ 
    denote the $k$-th non-zero eigenvalue of the Laplacian on $(X_s, \omega^{\rm CY}_s)$.

    Since $(X_s, \omega^{\rm CY}_s)$ is a compact Riemannian manifold with non-negative Ricci curvature, classical results of Cheng 
    \cite{Cheng1975}*{Corollary~2.2} and Wu--Yang--Zhong \cite{Wu1991}*{Theorem~14.2} yield the bounds:
    \[
    \frac{1}{\diam(X_s, \omega^{\rm CY}_s)^2} \leq \lambda_k^{\rm CY}(s) \leq \frac{8k^2 n(n + 2)}{\diam(X_s, \omega^{\rm CY}_s)^2},
    \]
    where $n$ is the complex dimension of $X_s$ and $\diam(X_s, \omega^{\rm CY}_s)$ is the diameter of $(X_s, \omega^{\rm CY}_s)$.

    In light of recent developments in the geometry of degenerating Calabi--Yau manifolds, 
    the asymptotic behavior of $\lambda_k^{\rm CY}(s)$ depends entirely on the dimension 
    of the essential skeleton $\mathrm{Sk}(X)$ associated to the degeneration 
    (see \cites{Kontsevich-Soibelman, Mustata-Nicaise2015} for the definition). 
    Specifically, we have two cases:
    \begin{itemize}
        \item 
        If $\dim \mathrm{Sk}(X) = 0$, a result of Rong--Zhang \cite{Rong-Zhang2011}*{Theorem~1.4} 
        ensures that the diameter is uniformly bounded from above and below; that is, 
        there exists a constant $C > 0$ independent of $s$ such that 
        $C^{-1} \leq \diam(X_s, \omega^{\rm CY}_s) \leq C$. Consequently, there exist 
        uniform constants $C_1, C_2 > 0$ such that for all $k \geq 1$,
        \[
        C_1 \leq \lambda_k^{\rm CY}(s) \leq C_2 k^2.
        \]
        In this case, there are {no} small eigenvalues.
        
        \item 
        If $\dim \mathrm{Sk}(X) \geq 1$, a recent theorem of Li--Tosatti 
        \cite{Li-Tosatti2024}*{Theorem~1.1} establishes that the diameter grows 
        on the order of $\abs{\log\abs{s}}^{\frac{1}{2}}$. That is, for a uniform constant 
        $C > 0$, we have $C^{-1} \abs{\log\abs{s}}^{\frac{1}{2}} \leq \diam(X_s, \omega^{\rm CY}_s) \leq C \abs{\log\abs{s}}^{\frac{1}{2}}$. 
        It then follows that
        \[
        \frac{C_1}{\abs{\log\abs{s}}} \leq \lambda_k^{\rm CY}(s) \leq \frac{C_2 k^2}{\abs{\log\abs{s}}}
        \]
        for uniform constants $C_1, C_2 > 0$ and all $k \geq 1$. In this case, 
        each fixed non-zero eigenvalue tends to $0$.
    \end{itemize}
\end{example}

\begin{example}
    \label[example]{example:hyperbolic-curves}
    Let $\pi \colon X \to \D$ be a degenerating family
    of hyperbolic curves. We equip each smooth fiber $X_s$ 
    ($s \neq 0$) with its unique K\"ahler--Einstein
    (hyperbolic) metric 
    $\omega^{\rm hyp}_s \in c_1(K_{X_s})$. Let $\lambda_k^{\rm hyp}(s)$ 
    denote the $k$-th non-zero eigenvalue of the Laplacian on $(X_s, \omega^{\rm hyp}_s)$.

    Assume first that the degeneration is stable. By the results of 
    Schoen--Wolpert--Yau \cite{Schoen-Wolpert-Yau1980} and Masur \cite{Masur1976} (see 
    also \cite{Grotowski-Huntley-Jorgenson2001}), there are exactly
    $N - 1$ small eigenvalues satisfying the following asymptotic bounds:
    \[
    \frac{C_1}{\abs{\log \abs{s}}} 
    \leq \lambda_1^{\rm hyp}(s) 
    \leq \cdots \leq \lambda_{N - 1}^{\rm hyp}(s)
    \leq \frac{C_2}{\abs{\log \abs{s}}},
    \]
    where $N$ is the number of irreducible components
    of the central fiber $X_0$, and $C_1, C_2 > 0$ are constants independent of $s$.

    For a general degeneration of hyperbolic curves, similar asymptotics 
    can be deduced by passing to its stable model via the Deligne--Mumford 
    stable reduction theorem \cite{Deligne-Mumford1969}. We remark that the base change involved in stable 
    reduction induces isometries on the regular fibers, as the hyperbolic 
    metric is canonically determined by the complex structure.
\end{example}

\begin{remark}
    We expect similar eigenvalue asymptotics to hold for 
    degenerations of compact K\"ahler manifolds of general type 
    equipped with canonical K\"ahler--Einstein metrics. 
    This would generalize the behavior observed in 
    \Cref{example:hyperbolic-curves} to higher dimensions.
\end{remark}

\begin{remark}
    For a degenerating family of Fano manifolds equipped with canonical K\"ahler--Einstein metrics, there are no small eigenvalues. This follows immediately from the classical theorem of Lichnerowicz, which guarantees a uniform positive lower bound for the first non-zero eigenvalue of the Laplacian.
\end{remark}

\subsection{Non-Archimedean picture for 
{small eigenvalues}}
\label[subsection]{sec:NA-subsection}

We now explain how the appearance of {small eigenvalues} is
reflected in the non-Archimedean limit of a degeneration, in the sense
of hybrid convergence of Boucksom--Jonsson.

We first recall the relevant convergence results for complex
Monge--Amp\`ere measures. Let
$\pi\colon X \to \D$ be a projective degeneration of compact complex
manifolds, with central fiber $X_0$.
\begin{itemize}
    \item 
    Suppose that $X$ is a Calabi--Yau degeneration polarized by a
    relatively ample line bundle $L$. We equip each smooth fiber $X_s$
    ($s \neq 0$) with the unique Ricci-flat Calabi--Yau metric
    $\omega_{\mathrm{CY},s}$ representing
    $c_1(L)|_{X_s}$. Boucksom--Jonsson
    \cite{Boucksom-Jonsson2017} proved that the Monge--Amp\`ere
    measures $\omega_{\mathrm{CY},s}^n$ converge in $X^{\rm hyb}$ to a
    Lebesgue-type measure on the top-dimensional part of the essential
    skeleton $\mathrm{Sk}(X)$.

    \item
    Suppose that $X$ is a degeneration of compact complex manifolds of
    general type. We equip each smooth fiber $X_s$ ($s \neq 0$) with
    its unique K\"ahler--Einstein metric
    $\omega_{\mathrm{KE},s} \in c_1(K_{X_s})$. Pille-Schneider
    \cite{PilleSchneider2022}*{Theorem A} proved that the
    Monge--Amp\`ere measures $\omega_{\mathrm{KE},s}^n$ converge in
    $X^{\rm hyb}$ to a Dirac-type measure supported on the divisorial
    valuations corresponding to the irreducible components of the 
    central
    fiber of the canonical model $\mathcal{X}_c$ of $X$. More precisely,
    \[
        \mu =
        \sum_{D \in \mathrm{Irr}((\mathcal{X}_c)_0)}
        ((K_{\mathcal{X}_c})^n \cdot D)\delta_{v_D}.
    \]

    \item
    Let $\omega_{\mathrm{bg},s} \coloneqq \omega_X|_{X_s}$ be the
    restriction of a background K\"ahler metric. Pille-Schneider
    \cite{PilleSchneider2022}*{Theorem B} proved that the
    Monge--Amp\`ere measures $\omega_{\mathrm{bg},s}^n$ converge in
    $X^{\rm hyb}$ to a Dirac-type measure supported on the divisorial
    valuations corresponding to the irreducible components of 
    the central
    fiber of a normal model $\mathcal{X}$ of $X$ with reduced central
    fiber. More precisely,
    \[
        \mu =
        \sum_{D \in \mathrm{Irr}(\mathcal{X}_0)}
        \left(\int_D(\nu^*\omega_X)^n\right)\delta_{v_D},
    \]
    where $\nu\colon \mathcal{X}\to X$ is the induced map.
\end{itemize}

We next compare these non-Archimedean limits with the spectral behavior
of degenerating curves. We use two elementary spectral models for graphs. First, for a
finite connected graph with positive vertex masses, we choose positive
edge conductances $c_e$ (for instance $c_e=1$) and use the weighted graph
Laplacian on $\ell^2(V,\mu)$,
\[
    (\Delta_{\Gamma,\mu} f)(v)
    =
    \frac{1}{\mu(v)}
    \sum_{w\sim v} c_{vw}\bigl(f(v)-f(w)\bigr).
\]
This is the usual weighted graph Laplacian with vertex measure; see
\cite{Lange-Liu-Peyerimhoff-Post2015}*{Section~2, especially (2.3)--(2.7)}
with trivial signature. Its Dirichlet energy is
\[
    \sum_{\{v,w\}\in E} c_{vw}\abs{f(v)-f(w)}^2,
\]
so, since the graph is connected, the kernel consists exactly of the
constant functions. Hence, if the graph has $N$ vertices, the Laplacian
has one zero eigenvalue and $N-1$ positive eigenvalues.

Second, a compact metric graph is a finite graph whose edges are assigned
positive lengths, so that each edge is identified with a compact interval.
On such a graph, we use the standard Kirchhoff Laplacian: it is
$-d^2/dx^2$ on each edge, with continuity of functions and the Kirchhoff
condition at every vertex. Its spectrum is discrete and satisfies Weyl's
law, so it has infinitely many positive eigenvalues; see
\cite{Berkolaiko2017}*{Sections~2 and~4.1, especially Lemma~4.4}.

\begin{itemize}
    \item
    For degenerations of hyperbolic curves with K\"ahler--Einstein
    metrics $(X_s,\omega_{\mathrm{KE},s})$, the complex-analytic theory
    gives exactly $N-1$ small eigenvalues, each of order
    $\abs{\log\abs{s}}^{-1}$, where $N$ is the number of irreducible
    components of the central fiber of the canonical model
    $\mathcal{X}_c$; see \Cref{example:hyperbolic-curves}.

    On the non-Archimedean side, the measures
    $\omega_{\mathrm{KE},s}$ converge to a Dirac-type measure supported
    on the divisorial points corresponding to these irreducible
    components. Since we are working with a stable model, each component
    carries strictly positive mass. Thus the dual graph of
    $(\mathcal{X}_c)_0$, together with the vertex masses
    $v \mapsto \mu(v)$, becomes a weighted graph. Its weighted graph
    Laplacian has exactly $N-1$ positive eigenvalues. When $N=2$, Ji
    \cite{Ji1993} proved that the rescaled eigenvalue
    $\abs{\log\abs{s}}\,\lambda_1(s)$ converges to the corresponding
    graph-theoretic eigenvalue; see also
    \cite{Grotowski-Huntley-Jorgenson2001}*{Conjecture 9}. We expect
    the same spectral convergence for arbitrary $N$.

    \item
    For degenerations of complex curves with induced background metrics
    $(X_s,\omega_{\mathrm{bg},s})$, we take as normal model the space
    $\widehat{F^{-1}X}$ introduced in \Cref{diagram:diagram-Z}, and
    denote its central fiber by $Z$.

    The complex-analytic results give exactly $N-1$ small eigenvalues,
    each of order $\abs{\log\abs{s}}^{-1}$, where $N$ is the number of
    irreducible components of $Z$; see \Cref{thm:main-thm-combined}.
    Meanwhile, the non-Archimedean convergence theorem gives a
    Dirac-type limiting measure supported on the components of $Z$.
    Each of these components has strictly positive mass by
    \Cref{lem:counting-irrd-cpnt-of-Z}. Hence the dual graph of $Z$,
    weighted by $v \mapsto \mu(v)$, again has a weighted graph
    Laplacian with exactly $N-1$ positive eigenvalues.

    In the case $N=2$, Dai--Yoshikawa proved that 
    the
    rescaled limit 
    $\abs{\log\abs{s}}\,\lambda_1(s)$ exists (see 
    \cite{Dai-Yoshikawa2025}*{Problem 9.2}). We expect that, for general
    $N$, the rescaled small eigenvalues converge to the positive
    eigenvalues of the weighted graph Laplacian of the corresponding
    dual graph.

    \item
    For degenerations of elliptic curves polarized by a relatively
    ample line bundle $L$, we equip each smooth fiber $X_s$ ($s\neq 0$)
    with the unique Ricci-flat Calabi--Yau metric
    $\omega_{\mathrm{CY},s}$ representing $c_1(L)|_{X_s}$.

    The complex geometry gives two sharply different possibilities:
    there are either no small eigenvalues or infinitely many, depending
    on the dimension of the essential skeleton $\mathrm{Sk}(X)$; see
    \Cref{example:calabi-yau}. The non-Archimedean picture gives the
    same dichotomy. If $\dim \mathrm{Sk}(X)=0$, then
    $\mathrm{Sk}(X)$ is a point by connectedness, and the limiting
    measure is a Dirac measure. The corresponding graph has one vertex,
    so its weighted graph Laplacian has no positive eigenvalues.

    If $\dim \mathrm{Sk}(X)=1$, then after passing to a semistable
    model, the essential skeleton is a metric graph homeomorphic to
    $\S^1$. The limiting measure is Lebesgue-type on the edges, with
    positive density. The resulting metric graph Laplacian behaves like
    the Laplacian on a circle, and therefore has infinitely many
    positive eigenvalues.
\end{itemize}

\begin{remark}
    When \(X\) is a degeneration of compact hyperbolic curves equipped
    with the Arakelov--Bergman metrics, the corresponding
    non-Archimedean spectral convergence picture is established in the
    work of Amini--Nicolussi \cite{Amini-Nicolussi2022}. We also recall
    that the hybrid limit of the Arakelov--Bergman measures
    is the Zhang measure, as proved by Shivaprasad
    \cite{Shivaprasad2024} and by Amini--Nicolussi 
    \cite{Amini-Nicolussi2025AENS} independently.
\end{remark}

In summary, these examples suggest a spectral convergence principle for
{small eigenvalues} under hybrid convergence. On the complex side,
the scale \(\abs{\log\abs{s}}^{-1}\) appears for induced metrics in
\Cref{thm:main-thm-combined}, for Calabi--Yau metrics in
\Cref{example:calabi-yau}, and for hyperbolic metrics in
\Cref{example:hyperbolic-curves}. On the non-Archimedean side, the
convergence of the Monge--Amp\`ere measures \(\omega_s^n\) is known from
work of Boucksom--Jonsson in the Calabi--Yau setting, and from work of
Pille-Schneider for induced metrics and for general type degenerations.

In higher dimensions, several substantial pieces of this picture are
already available, including non-Archimedean pluripotential theory,
hybrid convergence of Monge--Amp\`ere measures, and metric convergence
results for special classes of Calabi--Yau degenerations. However, to
the best of the author's knowledge, a general canonical spectral object
on \(X^{\mathrm{an}}\), or on \(\mathrm{Sk}(X)\), together with a
convergence theorem for the rescaled Dirichlet forms and spectra, is not
yet available with the same level of generality as in the
one-dimensional theory of Amini--Nicolussi.

\appendix

\bibliographystyle{amsalpha}
\bibliography{refs}

\bigskip
\footnotesize

\textsc{Courant Institute of Mathematical Sciences, 
New York University, 251
Mercer St, New York, NY 10012
}
\par\nopagebreak
\textit{Email address}, \texttt{junyu.cao@nyu.edu
}\par\nopagebreak
\textit{Homepage}, \url{https://junyucao1024.github.io}.

\end{document}